\documentclass[11pt]{article}
\usepackage{amssymb,amsmath}
\usepackage{graphicx}
\usepackage{tocloft}

\newtheorem{theorem}{Theorem}[section]

\newtheorem{lemma}[theorem]{Lemma}

\newtheorem{conjecture}[theorem]{Conjecture}
\newtheorem{corollary}[theorem]{Corollary}
\newtheorem{question}[theorem]{Question}

\newcommand{\qed}{\relax\ifmmode\hskip2em\Box\else\unskip\nobreak\hfill$\Box$\fi}

\newcounter{claim}

\newenvironment{proof}[1][]%
	{\noindent {\setcounter{claim}{0}\it Proof. }{#1}{}}{$\square$\vspace{2ex}}

\newenvironment{claim}[1][]%
	{\refstepcounter{claim}\vspace{1ex}\noindent {(\it\arabic{claim}) {#1}{}}\it}{\vspace{1ex}}
\newenvironment{proofclaim}[1][]%
	{\noindent {}{#1}{}}{ This proves~(\arabic{claim}).\vspace{1ex}}

\title{Perfect graphs: a survey}
\author{Nicolas Trotignon\\CNRS, LIP, ENS de Lyon\\Email: nicolas.trotignon@ens-lyon.fr}

\begin{document}
\maketitle

\begin{abstract}
  Perfect graphs were defined by Claude Berge in the 1960s.  They are
  important objects for graph theory, linear programming and
  combinatorial optimization.  Claude Berge made a conjecture about
  them, that was proved by Chudnovsky, Robertson, Seymour and Thomas
  in 2002, and is now called the strong perfect graph theorem.  This
  is a survey about perfect graphs, mostly focused on the strong
  perfect graph theorem.


\end{abstract}

{\vspace{-2ex}}

\tableofcontents
\newpage

\section*{Addendum}
\addcontentsline{toc}{section}{Addendum}

A short version of this survey appeared in~\cite{nicolas:pefectBook}.

Since the first publication of this survey, an open problem has been
solved~: Conjecture~\ref{c:oddHole}, solved by Scott and
Seymour~\cite{scottSey:oddHole}.  

The author of the present survey does not plan future versions.

\section{Introduction}

The \emph{chromatic number} of a graph $G$, denoted by $\chi(G)$, is
the minimum number of colours needed to assign a colour to each vertex
of $G$ in such a way that adjacent vertices receive different colours.
The \emph{clique number} of $G$, denoted by $\omega(G)$ is the maximum
number of pairwise adjacent vertices in $G$.  Every graph $G$ clearly
satisfies $\chi(G) \geq \omega(G)$, because the vertices of a clique
must receive different colours.  A graph $G$ is \emph{perfect} if
every induced subgraph $H$ of $G$ satisfies $\chi(H) = \omega(H)$.  A
chordless cycle of length $2k+1$, $k\geq 2$, satisfies $3 = \chi >
\omega = 2$, and its complement satisfies $k+1 = \chi > \omega = k$.
These graphs are therefore \emph{imperfect}.  Since perfect graphs are
closed under taking induced subgraphs, they must be defined by
excluding a familly $\cal F$ of graphs as induced subgraphs.  The
\emph{strong perfect graph theorem} (\emph{SPGT} for short) states
that the two examples that we just gave are the only members of $\cal
F$.  

Let us make this more formal.  A \emph{hole} in a graph $G$ is an
induced subgraph of $G$ isomorphic to a cycle chordless cycle of
length at least~4.  An \emph{antihole} is an induced subgraph $H$ of
$G$, such that $\overline{H}$ is hole of $\overline{G}$.  A hole
(resp.\ an antihole) is \emph{odd} or \emph{even} according to the
number of its vertices (that is equal to the number of its edges).  A
graph is \emph{Berge} if it does not contain an odd hole nor an odd
antihole.  The following, known as the \emph{SPGT}, was conjectured by
Berge~\cite{berge:61} in the 1960s and was the object of much research
until it was finally proved in 2002 by Chudnovsky, Robertson, Seymour
and Thomas~\cite{chudnovsky.r.s.t:spgt} (since then, a shorter proof
was discovered by Chudnovsky and
Seymour~\cite{chudnovsky.seymour:even}).

\begin{figure}[h]
\center  \parbox[c]{1.3cm}{\includegraphics{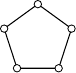}} $\equiv$
  \parbox[c]{1.3cm}{\includegraphics{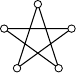}} \rule{1.5em}{0ex} 
  \parbox[c]{1.3cm}{\includegraphics{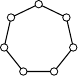}} \rule{1.5em}{0ex} 
  \parbox[c]{1.3cm}{\includegraphics{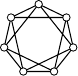}} $\equiv$ 
  \parbox[c]{1.4cm}{\includegraphics{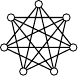}} 
\caption{Odd holes and antiholes:  $C_5 \equiv \overline{C_5}$, $C_7$ and $\overline{C_7}$}
\end{figure}

\begin{theorem}[{Chudnovsky, Robertson,
  Seymour and Thomas 2002}]
  \label{th:spgt}
  A graph is perfect if and only if it is Berge. 
\end{theorem}

One direction is easy: every perfect graph is Berge, since as we
observed already odd holes and antiholes satisfy $\chi = \omega + 1$.
The proof of the converse is very long and relies on structural graph
theory.  The main step is a \emph{decomposition theorem}
(Theorem~\ref{th:dec}), stating that every Berge graph is either in a
well-understood \emph{basic} class of perfect graphs (see
Section~\ref{s:basic}), or has some \emph{decomposition} (see
Section~\ref{s:dec}).  The goal of this survey is to make this result
and its meaning understandable to a non-specializist mathematician.  To
this purpose, not only the proof is surveyed, but also results that
were seminal to it, and some that were proved after it.  Our goal is
also to present perfect graphs as a lively subject for researchers, by
mentioning several attractive open questions at the end of the coming
sections. Putting forward up to date open questions is useful, because
the proof of the SPGT somehow took back the motivation for many
questions, in particular about classes of perfect graphs whose study
was supposed to give insight about the big open question.

Let us start now with an open question that obviously has the same
flavour as the SPGT. Gy\'arf\'as~\cite{gyarfas:perfect} proposed the
following generalization of perfect graphs.  A graph is
\emph{$\chi$-bounded by a function $f$} if every induced subgraph $H$
of $G$ satisfies $\chi(H) \leq f(\omega(H))$.  Hence, perfect graphs
are $\chi$-bounded by the identity function.  There might exist a
short proof that for some function $f$ (possibly fast increasing),
all Berge graphs are $\chi$-bounded by $f$, but so far, the proof of
the SPGT is the only known proof of this fact.
In~\cite{gyarfas:perfect} the following conjecture is stated.  It is
still open, but several attempts led to beautiful results, such as
Scott~\cite{Scott99a} and Chudnovsky, Robertson, Seymour and
Thomas~\cite{chudnovsky.r.s.t:k4}.

\begin{conjecture}[Gy\'arf\'as 1987]
  \label{c:oddHole}
  There exists a function $f$ such that for all graphs $G$, if $G$ has
  no odd hole, then $G$ is $\chi$-bounded by $f$.
\end{conjecture}

\subsection*{Terminology}

By the SPGT, perfect graphs and Berge graphs form the same
class. Still, we use \emph{Berge} when we mostly rely on excluding
holes and antiholes, and \emph{perfect} when we mostly rely on
colourings.  Anyway, this distinction has to be kept when we sketch
the proof of the SPGT.  A \emph{minimally imperfect graph} is an
imperfect graph every proper induced subgraph of which is perfect.  A
restatement of the SPGT is that minimally imperfect graphs are
precisely the odd holes and antiholes.  Many statements about minimal
imperfect graphs are therefore trivial to check by using the SPGT, but
when proving the SPGT, it is essential to prove them by other
means. Observe that a minimal (or minimum) counter-example to the SPGT
has to be a Berge minimally imperfect graph.

We mostly follow the terminology from~\cite{chudnovsky.r.s.t:spgt}
that is sometimes not fully standard.  Since all this theory deals
with induced subgraphs, we write \emph{$G$ contains $H$} to mean that
$H$ is an induced subgraph of $G$.  We simply write \emph{path}
instead of chordless path or induced path.  When $a$ and $b$ are
vertices of a path $P$, we denote by $aPb$ the subpath of $P$ whose
ends are $a$ and $b$.  A subset $A$ of $V(G)$ is \emph{complete} to a
subset $B$ of $V(G) $ if $A$ and $B$ are disjoint and every vertex of
$A$ is adjacent to every vertex of $B$ (we also say that $B$ is
\emph{$A$-complete}).

We use the prefix \emph{anti} to mean a property or a structure of the
complement (like in holes and antiholes).  For instance, an
\emph{antipath} in $G$ is a path in $\overline{G}$; $A$ is
\emph{anticomplete} to $B$ means that no edge of $G$ has an end in $A$
and the other one in $B$; a graph $G$ is \emph{anticonnected} if its
complement is connected.  By $C(A)$ we mean the set of all
$A$-complete vertices in $G$ and by $\overline{C}(A)$ the set of all
$A$-anticomplete vertices.

By \emph{colouring} of a graph, we mean an optimal colouring of the
vertices. 

\subsection*{Further reading}

Since we focus on the SPGT and its proof, most of this survey is on
the structure of Berge graphs, and we omit many other important (more
important?)  aspects of the theory of perfect graphs, such as the
ellipso\"id method used by Gr\"ostchel, Lov{\'a}sz and
Schrijver~\cite{grostchel.l.s:color} to colour in polynomial time any
input perfect graph.  The theory of semi-definite programming started
there.

Several previous surveys exist (and we try not to overlap them too
much).  As far as we know, all of them were written before of just
after the proof of the SPGT.  The survey of
Lov\'asz~\cite{lovasz:pgsurvey} from the 1980s is still a good
reading.  Two books are completely devoted to perfect
graphs~\cite{berge.chvatal:topics,livre:perfectgraphs}, and contain a
lot of material that will be cited in what follows.  Part VI of the
treatise of Schrijver~\cite{schrijver:opticomb} is a good survey on
perfect graphs (it is the most comprehensive).  Chudnovsky, Robertson,
Seymour and Thomas~\cite{chudvovsky.r.s.t:progress} wrote a good
survey just after their proof.  Roussel, Rusu and
Thuillier~\cite{RousselRT09} wrote a good survey about the long
sequence of results, attempts and conjectures that finally led to the
successful approach.

About more historical aspects, see Section 67.4g in
Schrijver~\cite{schrijver:opticomb}.  Berge and Ram{\'{i}}rez
Alfons{\'{i}}n~\cite {berge.r:origin} wrote an article about the
origin of perfect graphs.  Seymour~\cite{seymour:how} wrote an
article, that is interesting even to a non-mathematician, telling the
story of how the SPGT was proved.

\section{Lov\'asz's perfect graph theorem}
\nocite{livre:perfectgraphs}

As pointed out by Preissmann and Seb\H
o~\cite{preissmann.sebo:minimal}, the following conveniently gives a
weak \ref{i:f2} and a strong \ref{i:f3} characterization of perfect
graphs.  Hence, to prove perfection, checking the weak one is enough,
while to use perfection, one may rely on the strong one.  This allows
a kind of leverage that is used a lot in the sequel.

\begin{lemma}[Folklore]
  \label{l:simpleC}
  Let $G$ be a graph.  The three statements below are equivalent.
  \begin{enumerate}
  \item\label{i:f1} $G$ is perfect.
  \item\label{i:f2} For every induced subgraph $H$ of $G$ and every
    $v\in V(H)$, $H$ contains a stable set that contains $v$ and intersects
    all maximum cliques of $H$.
  \item\label{i:f3} For every induced subgraph $H$ of $G$, $H$ contains a
    stable set that intersects all maximum cliques of $H$.
  \end{enumerate}
\end{lemma}

\begin{proof}
  To prove that \ref{i:f1} implies \ref{i:f2}, consider an optimal
  colouring of $H$ and the colour class $S$ that contains $v$.  So $S$
  is a stable set, and it must intersect all maximum cliques of $H$
  for otherwise $\chi(H\setminus S) \geq \omega(H)$, a contradiction.
  Trivially, \ref{i:f2} implies \ref{i:f3}.  To prove that \ref{i:f3}
  implies \ref{i:f1}, consider the following greedy colouring
  algorithm: (step 1) set $i \leftarrow 1$ and $H \leftarrow G$; (step 2) while
  $H$ is non-empty, consider a stable set of $H$ as in \ref{i:f3},
  give it colour $i$, and set $H \leftarrow H \setminus S$, $i \leftarrow
  i+1$.  Since $\omega$ decreases at each step, this algorithm
  produces a colouring of $G$ with $\omega(G)$ colours, and it can be
  processed for any induced subgraph of $G$.  Thus, $G$ is perfect.
\end{proof}

\begin{figure}[h]
  \center
  \includegraphics{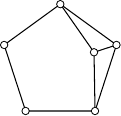}\hspace{3em}
  \includegraphics{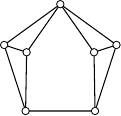}
  \caption{Replication: $G$ and $G'$\label{fig:rep}}
\end{figure}

\emph{Replicating} a vertex $v$ of a graph $G$ means adding a new
vertex $v'$ adjacent to $v$ and all neighbors of $v$.  As an example,
consider the (non-perfect) graph $G$ obtained by replicating one
vertex of $C_5$.  Clearly, $\chi(G) = \omega(G) = 3$ (see
Fig.~\ref{fig:rep}).  However, by replicating any vertex of degree~2
in $G$, a graph $G'$ such that $\chi(G') = 4 > \omega(G') = 3$ is
obtained.  This shows that the property $\chi = \omega$ is not
preserved by replication.  Therefore, the following lemma due to
Lov\'asz~\cite{lovasz:nh,lovasz:nh2} and known as the
\emph{replication lemma} is more surprising than it may look at first
glance.

\begin{lemma}[Lov\'asz 1972]
  \label{l:rep}
  Perfection is preserved by replication.
\end{lemma}

\begin{proof}
  Let $G$ be a perfect graph and $G'$ a graph obtained from $G$ by
  replicating a vertex $v$.  We show that $G'$ satisfies the
  characterization \ref{i:f3} from Lemma~\ref{l:simpleC}.  Let $H$ be
  an induced subgraph of $G'$.  We look for a stable set $S$ that
  intersects all maximum cliques of $H$.  If $H$ contains at most one
  vertex from $\{v, v'\}$, then it is isomorphic to a an induced
  subgraph of $G$, so clearly $S$ exists.  Otherwise, $H\setminus v'$
  is perfect, so, by charactization \ref{i:f2}, there exists a stable
  set $S$ that contains $v$ and intersects all maximum cliques of $H
  \setminus v'$.  In fact $S$ intersects all maximum cliques of $H$,
  because a maximum clique of $H$ contains $v$ if and only if it
  contains $v'$.
\end{proof}

The following was conjectured by Berge as the \emph{weak perfect graph
  conjecture} and is now called the \emph{perfect graph theorem}.  We
give the proof that, among the available ones, we feel the most
related to the rest of this survey.  The \emph{stable set number} of a
graph $G$, denoted by $\alpha(G)$, is the maximum number of pairwise
non-adjacent vertices.  Observe that for all graphs $G$, $\alpha(G) = \omega
(\overline{G})$.  

\begin{theorem}[Lov\'asz 1972]
  \label{th:pgt}
  If a graph is perfect, then its complement is perfect. 
\end{theorem}

\begin{proof}
  Let $G$ be a perfect graph.  Construct $G'$ as follows: start from
  $G$ and replicate $\alpha_v - 1$ times every vertex $v$, where
  $\alpha_v$ is the number of maximum stables sets of $G$ that contain
  $v$.  Note that replicating $-1$ time means deleting the vertex,
  and replicating $0$ times means doing nothing.  From its
  construction, $G'$ can be covered by $k$ disjoint maximum stable
  sets, that therefore form an optimal colouring of $G'$.  Since $G'$
  is perfect by Lemma~\ref{l:rep}, if follows that $G'$ has a clique
  $K'$ of size $k$.  Since a clique and a stable set intersect in at
  most one vertex, $K'$ intersect all maximum stable sets of $G'$.
  Now construct a clique $K$ of $G$, by taking for each vertex of $K'$
  the vertex of $G$ it is replicated from.  The clique $K$ that we
  obtain intersects all maximum stable sets of $G$.  By the same lines,
  a clique intersecting all maximum stable sets can be found in any
  induced subgraph of $G$.  Hence the complement of $G$ satisfies
  condition \ref{i:f3} of Lemma~\ref{l:simpleC} and is therefore
  perfect.
\end{proof}

In \cite{gyarfas:perfect}, the question of generalizing
Theorem~\ref{th:pgt} is discussed.  Since~\cite{gyarfas:perfect},
hardly any progress occurred in this direction.  In particular the
following neat generalization of Theorem~\ref{th:pgt} is still open.

\begin{conjecture}[Gy\'arf\'as 1987]
  There exists a function $f$ such that for all graphs $G$, if $G$ is
  $\chi$-bounded by $x\mapsto x+1$, then its complement is
  $\chi$-bounded by $f$.
\end{conjecture}

\subsection*{Further reading}

The perfect graph theorem has a polyhedral proof found by
Fulkerson~\cite{fulkerson:antiblocking} related to polyhedral
characterizations of perfect graphs discovered by
Fulkerson~\cite{fulkerson:antiblocking} and
Chv\'atal~\cite{chvatal:75}.  Lov\'asz~\cite{lovasz:pg} proved a deep
characterization of perfect graphs suggested by Hajnal: a graph $G$ is
perfect if and only if every induced subgraph $H$ of $G$ satisfies
$\alpha(H) \omega(H) \geq |V(H)|$ (this can be proved by a simple
argument relying on linear algebra discovered by Gasparian
\cite{gasparian:minimp}, see also
\cite{diestel:graph,bondy.murty:book}, and Section~\ref{sec:col} below
for applications of linear algebraic methods to perfect graphs).  This
characterization implies that deciding the perfection of an input
graph is a CoNP problem (see \cite{preissmann.sebo:minimal} for more
about that).  Since the characterization is self-complementary, it
gives another proof of the perfect graph theorem.  This
characterization is the starting point of many developments of great
signifiance, such as the theory of partitionable graphs (see the
survey of Preissmann and Seb\H o~\cite{preissmann.sebo:minimal}).  It
has deep connections with combinatorial optimization as explained in a
book of Cornu\'ejols~\cite{cornuejols:packing}.

\section{Basic graphs}
\label{s:basic}

In this section, we survey the five basic classes that are used in the
proof of the SPGT.  We denote by $\theta(G)$ the chromatic number of
$\overline{G}$, by $\nu(G)$ the maximum size of a matching in $G$, by
$\Delta(G)$ the maximum degree of a vertex in $G$, by $\tau(G)$ the
minimum number of edges of $G$ needed to cover all vertices of $G$,
and by $\chi'(G)$ the minimum number of colours needed to assign a
colour to each \emph{edge} of $G$ in such a way that adjacent edges
receive different colours.  Bipartite graphs are easily checked to be
perfect.  So, by Theorem~\ref{th:pgt}, their complements are also
perfect, which can be restated as `every bipartite graph $G$ satisfies
$\theta(G) = \alpha(G)$'.  Since for any triangle-free graph $G$,
$|V(G)| = \theta(G) + \nu(G) = \alpha(G) + \tau(G)$, we obtain that
every bipartite graph $G$ satisfies $\nu(G) = \tau(G)$.  This can be
rephrased as: `the complements of line graphs of bipartite graphs are
perfect'.  By applying again Theorem~\ref{th:pgt}, we obtain that line
graphs of bipartite graphs are perfect, which can be restated as
`every bipartite graph $G$ satisfies $\Delta(G) = \chi'(G)$'.  Hence,
Theorem~\ref{th:pgt} implies the perfection of three among the four
\emph{historical basic} classes of perfect graphs: bipartite graphs,
their complements, their line graphs, the complements of their line
graphs.  Interestingly, this was all proved directly by K\H
onig~\cite{konig:31,konig:16}, long  before the definition of
perfect graphs.

We now turn our attention to a less classical class that is first
presented in the proof of the SPGT: \emph{double split graphs}.  As it
is presented in~\cite{chudnovsky.r.s.t:spgt}, the class is not closed
under taking induced subgraphs, which is sometimes not convenient.  So
we prefer here to define directly \emph{doubled graphs}, that are
easily seen to form the class of induced subgraphs of double split
graphs (as defined in~\cite{chudnovsky.r.s.t:spgt}).

A \emph{good partition} of a graph $G$ is a partition $(X, Y)$ of
$V(G)$ (possibly, $X=\emptyset$ or $Y=\emptyset$) such that:

\begin{itemize}
\item Every component of $G[X]$ has at most two vertices, and every
  anticomponent of $G[Y]$ has at most two vertices.
\item For every component $C_X$ of $G[X]$, every anticomponent $C_Y$ of
  $G[Y]$, and every vertex $v$ in $C_X \cup C_Y$, there exists at most
  one edge and at most one antiedge between $C_X$ and $C_Y$ that is
  incident to $v$.
\end{itemize}

A graph is \emph{doubled} if it has a good partition (for the sake of
completeness, let us mention that a \emph{double split graph} is a
doubled graph such that $G[X]$ (resp.\ $G[Y]$) has at least two
components (resp. anticomponents) and all components (resp.\
anticomponents) of $G[X]$ (resp. $G[Y]$) have two vertices).  Doubled
graphs are easily seen to be perfect by a direct colouring argument.
They are closed under taking induced subgraphs and complements.  A
graph is \emph{basic} if it belongs to at least one of the five
classes defined here: bipartite, complement of bipartite, line graph
of bipartite, complement of line graph of bipartite, and doubled
graphs.  

For each basic class, the characterization by excluding induced
subgraphs is known (see Beineke~\cite{beineke:linegraphs} for line
graphs and Alexeev, Fradkin and Kim~\cite{alexeevFK:doubled} for
doubled graphs); the recognition can be performed in polynomial time
(see Lehot~\cite{lehot:root} or
Roussopoulos~\cite{roussopoulos:linegraphe} for line graphs and
Maffray~\cite{maffray:13} for doubled graphs).  Also the colouring and
maximum clique problems can be solved in polynomial time (see
Schrijver~\cite{schrijver:opticomb} for the `historical classes' and
Maffray~\cite{maffray:13} for doubled graphs).  Note that the paper of
Maffray~\cite{maffray:13} answers a question asked in a previous
version of the present survey.

\subsection*{Further reading}

In this section, we focused on the basic graphs that play an important
role in the proof of the SPGT.  But any class of graphs whose
perfection is simple to prove can potentially serve as a basic class
of a decomposition theorem, so all classes are potentially of
interest.  The book of Brandst\"adt, Le and
Spinrad~\cite{brandstadt:classes} is on general graph classes, but
contains a lot of material on perfect graphs.  The most complete
catalog of classes of perfect graphs seems to be written by
Hougardy~\cite{hougardy:class}, that describes 120 classes.  Also
Chapter~66 in Schrijver~\cite{schrijver:opticomb} contains a very
complete survey about classes of perfect graphs.  The book by
Golumbic~\cite{golumbic:perfect} surveys algorithmic aspects of
several classes of perfect graphs.

Perhaps the most important class that we omit to present here is the
seminal class of hole-free graphs, known as \emph{chordal} graphs and
introduced by Dirac~\cite{dirac:chordal} and
Gallai~\cite{gallai:triangule}.  It is the first class with a
decomposition theorem, has some connections with the graph minors
theory and tree-width, and is also important in fast graph searching
algorithms.  About that, a good starting point is Sections 9.7--9.8 in
the book of Bondy and Murty~\cite{bondy.murty:book}.  Another
important class, that has connections with ordered sets, is the class
of comparability graphs introduced by
Gallai~\cite{gallai:comparabilite} (the English translation by Maffray
and Preissmann~\cite{maff.preis:gallai} contains a short survey).

\section{Decompositions}
\label{s:dec}

By \emph{decomposition} of a graph we mean a way to partition its
vertices with some prescribed adjacencies.  A decomposition is
\emph{useful} if it can be proved that a minimum counter-example to the
SPGT cannot admit the decomposition.  Indeed, suppose that we can
prove a statement such as: `every Berge graph is either basic or has
a useful decomposition'.  The SPGT can then be proved as follows:
consider a minimum counter-example, i.e.\ a Berge graph, imperfect, and
of minimum size. It does not admit the decomposition (because the
decomposition is useful), and since it is imperfect, it cannot be
basic, a contradiction to the statement.

The simplest useful decomposition, first observed in this context by
Gallai~\cite{gallai:triangule}, is the \emph{clique cutset}, that is a
clique whose removal yields a disconnected graph. It is easily seen to
be useful, because gluing two perfect graphs along a clique yields a
perfect graph.

In a graph $G$, \emph{substituting} a graph $H$ for a vertex $v$,
means deleting $v$, adding a copy of $H$, making every neighbor of $v$
complete to $H$, and every non-neighbor of $v$ anticomplete to $H$.
Along the lines of the proof of Lemma~\ref{l:rep}, it is easy to
prove that substituting a perfect graph for a vertex of a perfect
graph yields a perfect graph (this is therefore a variant of Lov\'asz's
replication lemma).  It follows easily that a minimally imperfect
graph does not admit a homogeneous set, where a \emph{homogeneous set}
of a graph $G$ is a set $H \subseteq V(G)$ such that $1 < |H| <
|V(G)|$ and every vertex of $G\setminus H$ is either complete or
anticomplete to $H$.

A \emph{1-join} of a graph $G$, first defined by
Cunningham~\cite{cunningham:82}, is a partition $(X, Y)$ of $V(G)$
such that $|X| \geq 2$, $|Y| \geq 2$, and there exist $A\subseteq X$
and $B \subseteq Y$ such that $A$ is complete to $B$ and no other
edges exist from $X$ to $Y$.  Again, 1-joins can be proved to be
useful (this is proved by Cunningham~\cite{cunningham:82},
Bixby~\cite{bixby:84} and it also follows from Lemma~\ref{th:sc}
below).

Note that clique cutsets, homogeneous sets and 1-joins do not appear
explicitly in Theorem~\ref{th:dec}.  This is because they are not
formally necessary, since their presence implies that the graph is
basic or has another decomposition (namely the balanced skew
partition, to be defined soon).  We mention them because they are
somehow present `implicitly'; this sometimes shows up naturally in
attempts to use Theorem~\ref{th:dec} for algorithmic purpose.  We now
turn our attention to the decompositions that are actually used in
Theorem~\ref{th:dec}.  For each definition, we use the definition
from~\cite{chudnovsky.r.s.t:spgt}, that sometimes differ slightly from
the definition given in the paper where the decomposition is first
presented (and where the proof of the usefulness of the decomposition
is given).

A \emph{2-join} of a graph $G$, first defined by Cornu\'ejols and
Cunningham~\cite{cornuejols.cunningham:2join}, is a partition $(X_1,
X_2)$ of $V(G)$ such that there exist disjoint non-empty sets $A_1,
B_1 \subseteq X_1$, $A_2, B_2 \subseteq X_2$ satisfying:

\begin{itemize}
\item $A_1$ is complete to $A_2$, $B_1$ is complete to $B_2$ and these
  edges are the only ones between $X_1$ and $X_2$;
\item $|X_i| \geq 3$, $i=1, 2$;
\item every component of $G[X_i]$ intersects $A_i$ and $B_i$, $i=1,
  2$; and
\item if $|A_i| = |B_i| = 1$, then $G[X_i]$ is not a path of length
  two joining the members of $A_i$ and $B_i$, $i=1, 2$.
\end{itemize}

Cornu\'ejols and Cunningham~\cite{cornuejols.cunningham:2join} proved
that a minimally imperfect graph admiting a 2-join must be an odd hole
(so, 2-joins are useful).  A reader who pays attention to
technicalities may notice that here, \emph{path
  2-joins\label{pth2join}} are allowed (these are 2-joins such that
for some $i\in \{1, 2\}$, $|A_i|=|B_i|=1$ and $G[X_i]$ is a path from
the unique vertex in $A_i$ to the unique vertex in $B_i$).  Some
papers (mostly, these cosigned by Conforti, Cornu\'ejols or Vu\v
skovi\'c) restrict the notion of 2-joins to \emph{non-path 2-joins}.
In Section~\ref{sec:trigraphs}, the relevance of excluding path
2-joins is discussed.  When $X_i$, $A_i$ and $B_i$ are as above, it is
customary to set $C_i = X_i \setminus (A_i \cup B_i)$.  It is easy to
prove that in Berge graphs, all paths from $A_i$ to $B_i$ with
interior in $C_i$ have the same parity (otherwise, an odd hole
exists).  Therefore, there are two kinds of 2-joins, according to this
parity: \emph{odd} and \emph{even} 2-joins.  If $(X_1, X_2)$ is a
2-join of $\overline{G}$, then it is a \emph{complement 2-join} of
$G$.

When $G$ is a graph and $A\subseteq V(G)$, we denote by $C(A)$ the
sets of vertices of $G$ complete to $A$ and by $\overline{C}(A)$ the
set of vertices of $G$ anticomplete to $A$.  A \emph{homogeneous pair}
(first defined in a slightly different way by Chv\'atal and
Sbihi~\cite{chvatal.sbihi:bullfree} who proved that they are useful)
is a pair of disjoint sets $A, B \subseteq V(G)$ such that $|A|, |B|
\geq 2$, every vertex of $A$ has a neighbor and a non-neighbor in $B$,
every vertex of $B$ has a neighbor and a non-neighbor in $A$, and $A$,
$B$, $C(A) \cap \overline{C}(B)$, $\overline{C}(A) \cap C(B)$, $C(A)
\cap C(B)$, $\overline{C}(A) \cap \overline{C}(B)$ are all non-empty
and partition $V(G)$.

All the decompositions presented so far are nice in the following
sense.  When applied recursively,  they yield decomposition trees of
polynomial size that allow solving several problems.  The machinery is
too heavy to be presented here, see Section~\ref{sec:col}.  We now
turn our attention to other kinds of cutset that do not have this nice
property.

A \emph{star cutset} (first defined by
Chv\'atal~\cite{chvatal:starcutset}) in a graph $G$ is a set $S$ of
vertices such that $G \setminus S$ is disconnected and $S$ contains a
vertex $v$, called the \emph{center}, complete to $S \setminus v$.  The
following is known as the \emph{star cutset
  lemma}~\cite{chvatal:starcutset}.  It remarkably generalizes the
usefulness of clique cutsets, 1-joins and homogeneous sets (because
these decompositions all imply the presence of a star cutset, in the
graph or in some degenerate small cases, in the complement).

\begin{lemma}[Chv\'atal 1985]
  \label{th:sc}
  A minimally imperfect graph has no star cutset.
\end{lemma}

\begin{proof}
  Let $G$ be a minimally imperfect graph and suppose for a
  contradiction that $G$ has a star cutset $S$ centered at $v$.  Let
  $(X, Y)$ be a partition of $V(G \setminus S)$ such that $|X|, |Y|
  \geq 1$ and $X$ is anticomplete to $Y$.  We now prove that $G$
  satisfies the condition~\ref{i:f3} from Lemma~\ref{l:simpleC} (this
  implies that $G$ is perfect, giving the contradiction).  Since every
  proper induced subgraph of $G$ is perfect, it just remains to find
  the desired stable set in $G$.  By condition~\ref{i:f2} from
  Lemma~\ref{l:simpleC}, there exists a stable set $A_X$ in $G[S \cup
  X]$ that contains $v$ and intersects all maximum cliques of $G[S
  \cup X]$.  A similar stable set $A_Y$ exists in $G[S \cup Y]$.  Now,
  $A_X \cup A_Y$ is a stable set of $G$ that intersects all maximum
  cliques of $G$.
\end{proof}

It is quite easy to turn the proof above into a colouring algorithm
(that would for instance colour any perfect graph every induced
subgraph of which is either basic or decomposable by a star cutset).
The algorithm would output a stable set that intersects all maximum
cliques, and along the lines of the proof of Lemma~\ref{l:simpleC},
this gives a colouring algorithm.  Unfortunately, this algorithm does
not run in polynomial time, because the star cutset can be very big,
for instance be the entire graph except two vertices.  In this case,
the complexity analysis of the recursive calls leads to an exponential
number of calls.  There is a similar problem with the generalizations
that we consider now.  This is the main reason why the decomposition
of Berge graphs does not lead to a polynomial time colouring
algorithm.

A \emph{skew partition} (first defined by
Chv\'atal~\cite{chvatal:starcutset}) of a graph $G$ is a partition
$(A, B)$ of $V(G)$ such that $G[A]$ is not connected, and $G[B]$ is
not anticonnected.  In this case, we say that $B$ is a \emph{skew
  cutset}.  Following a prophetic insight that some
\emph{self-complementary} decomposition generalizing the star cutset
should play some role, Chv\'atal conjectured that a minimally
imperfect graph has no skew partition, and a less formal statement,
that skew partitions should appear in the decomposition of Berge
graphs.  Observe that if a graph on at least 5 vertices with at least
one edge has a star cutset, then it has a skew partition.  The proof
of Ch\'vatal's conjectures escaped the researchers, but several
fruitful attempts were made.  In particular, many special kinds of
skew partitions were proved not to be in minimal imperfect graphs (see
Reed~\cite{reed:skewhist} for a survey).  In the opposite direction, a
generalization of skew partitions was proved to decompose all Berge
graphs in the following theorem~\cite{conforti.c.v:dstrarcut}.  A
\emph{double star cutset} in a graph $G$ is a set $S \subseteq V(G)$
such that $G\setminus S$ is disconnected and $S$ contains an edge $uv$
such that every vertex of $S$ is adjacent to at least one of $u, v$.
Note that for a Berge graph $G$, the following gives in fact two
pieces of information: one for $G$, one for $\overline{G}$.

\begin{theorem}[Conforti, Cornu\'ejols and Vu\v skovi\'c 2004]
  \label{th:double}
  A graph with no odd hole is either basic, or has a 2-join or a
  double star cutset.
\end{theorem}

One of the breakthroughs made in the proof of the SPGT is the concept
of \emph{balanced skew partition}.  For a graph $G$, a partition (skew
or not) $(A, B)$ of $V(G)$ is \emph{balanced} if every path of
length at least~3, with ends in $B$ and interior in $A$, and every
antipath of length at least~3, with ends in $A$ and interior in $B$
has even length.  It is straightforward to check that a partition $(A,
B)$ of a Berge graph is balanced if and only if adding a vertex
complete to $B$ and anticomplete to $A$ yields a Berge graph.

As we will see, the notion of balanced skew partition is sufficiently
particular to allow a short proof that a minimum counter-examples to
the SPGT cannot contain it, and sufficiently general to be found in
all non-basic Berge graphs that cannot be decomposed otherwise.
Interestingly, Zambelli~\cite{zambelli:these} notices that if a Berge
graph on at least five vertices and with at least one edge has a star
cutstet, then it has a balanced skew partition.

\begin{lemma}
  \label{l:ext}
  If $(A, B)$ is a balanced partition of a perfect graph $G$, and if
  every Berge graph of size at most $|V(G)|+1$ is perfect, then $G[B]$
  admits a colouring that can be extended to a colouring of $G$.
\end{lemma}

\begin{proof}
  Consider the graph $G'$ obtained by adding a clique of size $k =
  \omega(G) - \omega(G[B])$ complete to $B$ and anticomplete to $A$.
  It is Berge because $(A, B)$ is balanced.  So it is perfect when $k
  \leq 1$ from our assumption, and it is also perfect and when $k\geq
  2$ by several applications of Lemma~\ref{l:rep}.  Observe that
  $\omega(G') = \omega(G)$.  An $\omega(G)$ colouring of $G'$ yields a
  colouring of $G[B]$ that extends to a colouring of $G$.
\end{proof}

\begin{theorem}[Chudnovsky, Robertson, Seymour and Thomas 2002]
  A minimum imperfect Berge graph admits no balanced skew partition.  
\end{theorem}

\begin{proof}
  Let $G$ be minimum imperfect Berge graph.  Hence, $\chi(G) >
  \omega(G)$.  Note that by Theorem~\ref{th:pgt}, $\overline{G}$ is
  also a minimum imperfect Berge graph.  Let $(A, B)$ be a balanced
  skew partition in $G$.  So, $A$ partitions into two sets $A_1$ and
  $A_2$ anticomplete to one another, and $B$ partitions into two sets
  $X$ and $Y$ complete to one another.  By Lemma~\ref{th:sc},
  $|A_1|, |A_2| \geq 2$, for otherwise, the unique vertex in $A_1$ or
  $A_2$ would be the center of a star cutset in $\overline{G}$.  From
  the minimality of $G$, it follows that every Berge graph of size
  $|V(G[B \cup A_i])| + 1$ is perfect, $i=1, 2$.  By
  Lemma~\ref{l:ext}, consider an $\omega(G[B])$ colouring $C_i$ of
  $G[B]$ that extends to a colouring of $G[B \cup A_i]$.  Let $X_i$ be
  the set of vertices of $G[B \cup A_i]$ whose colour in the colouring
  $C_i$ is present in $X$ and let $Y_i = (B \cup A_i) \setminus X_i$.
  Because of the colouring $C_i$, $\omega(G[X_i]) = \omega(G[X])$.  So
  $\omega(G[X_1 \cup X_2]) = \omega(G[X])$.  By the minimality of $G$,
  it follows that $G[X_1 \cup X_2]$ has an $\omega(G[X])$-colouring.
  Because of the colouring $C_i$, $\omega(G[Y_i]) = \omega(G) -
  \omega(G[X_i]) = \omega(G) - \omega(G[X])$.  So, $\omega(G[Y_1 \cup
  Y_2] = \omega(G) - \omega(G[X])$.  By the minimality of $G$, it
  follows that $G[Y_1 \cup Y_2]$ has an $(\omega(G) -
  \omega(G[X]))$-colouring.  It follows that $G$ has an
  $\omega(G)$-colouring, a contradiction.
\end{proof}

The following is maybe hopeless, because a proof would imply a direct
argument for the skew partition conjecture (no such argument exists
today).  Also a proof of the following together with
Theorem~\ref{th:double} would yield a new proof the SPGT.  Observe
that antiholes of length at least~6 have double star cutsets.  

\begin{question} 
  Find a direct proof of the following: if $G$ is a minimum Berge
  imperfect graph, then at least one of $G, \overline{G}$ admits no
  double star cutset.
\end{question}

\subsection*{Further reading}

Rusu~\cite{rusu:cutsets} wrote a survey about cutsets in perfect
graphs, see also~\cite{RousselRT09}.  Reed~\cite{reed:skewhist} wrote
a survey about skew partitions (on which this section is mostly
based).  It shows that many ideas of the proof presented above for
balanced skew partitions are
implicitly contained in several papers, namely  in 
Ho\`ang~\cite{hoang:minimp}, Olariu~\cite{Olariu90} and Roussel and
Rubio~\cite{roussel.rubio:01}.

An important question about decompositions is their detection in
polynomial time.  The following decompositions can all be detected in
polynomial time: clique cutset (in time $O(nm)$,
Tarjan~\cite{tarjan:clique}), homogeneous set (in time $O(n+m)$, see
Habib and Paul~\cite{HabibP10}), 1-join (in time $O(n+m)$, see Charbit,
de Montgolfier and Raffinot~\cite{ChMoRa:split}), 2-join (in time
$O(n^2m)$, Charbit, Habib, Trotignon and Vu\v
skovi\'c~\cite{ChHaTrVu:2-join}), homogeneous pair (in time
$O(n^2m)$, Habib, Mamcarz and de
Montgolfier~\cite{habibMamMon:Hjoin}), skew partitions (in time
$O(n^4m)$, Kennedy and Reed~\cite{kennedyreed:skew}).
Trotignon~\cite{nicolas:bsp} showed that balanced skew partition are
NP-hard to detect, but devised an $O(n^9)$-time non-constructive
algorithm that certifies whether an input \emph{Berge} graph has or
not a balanced skew partition.

\section{Truemper configurations}
\label{sec:truemper}

Before going further, we need to define several special kinds of
graphs, known as \emph{Truemper configurations}.  They appear in many
contexts (sometimes in older papers, such
as~\cite{watkinsMesner:cycle}).

A \emph{prism} is a graph made of three vertex-disjoint paths $P_1 =
a_1 \dots b_1$, $P_2 = a_2 \dots b_2$, $P_3 = a_3 \dots b_3$ of length
at least 1, such that $a_1a_2a_3$ and $b_1b_2b_3$ are triangles and no
edges exist between the paths except these of the two triangles.
Observe that a prism in a Berge graph must have the lengths of the
three paths of the same parity.  The prism is \emph{odd} or
\emph{even} according to this parity.
 
A \emph{pyramid} is a graph made of three paths $P_1 = a \dots b_1$,
$P_2 = a \dots b_2$, $P_3 = a \dots b_3$ of length at least~1, two of
which have length at least 2, vertex-disjoint except at $a$, and such
that $b_1b_2b_3$ is a triangle and no edges exist between the paths
except these of the triangle and the three edges incident to $a$.

A \emph{theta} is a graph made of three internally vertex-disjoint
paths $P_1 = a \dots b$, $P_2 = a \dots b$, $P_3 = a \dots b$ of
length at least~2 and such that no edges exist between the paths
except the three edges incident to $a$ and the three edges incident to
$b$.

\begin{figure}
   \begin{center}
     \includegraphics[height=2cm]{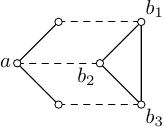}
     \hspace{1em}
     \includegraphics[height=2cm]{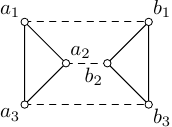}
     \hspace{1em}
     \includegraphics[height=1.8cm]{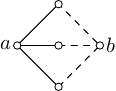}
     \hspace{1em}
     \includegraphics[height=2cm]{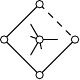}
   \end{center}
   \caption{Pyramid, prism, theta and wheel (dashed lines represent paths)\label{f:tc}}
 \end{figure}

Observe that the lengths of the paths in the three definitions are
designed so that the union of any two of the paths form a hole. 

A \emph{wheel} is a graph formed by a hole $H$ together with a vertex
that have at least three neighbors in the hole.   

A \emph{Truemper configuration} is a graph isomorphic to a prism, a
pyramid, a theta or a wheel.  As we will see, Truemper configurations
play a special role in the proof of the SPGT.  First, a Berge graph
has no pyramid (because among the three paths of a pyramid, two have
the same parity, and their union forms an odd hole).  This little fact
is used very often to provide a contradiction when working with Berge
graphs.  As we will soon see, a long part of the proof is devoted to
study the structure of a Berge graph that contains a prism, and another long
part is devoted to a Berge graph that contains a wheel.  And at the very
end of the proof, it is proved that graphs not previously decomposed
are bipartite, just as Berge thetas are.  Note also that prisms can be
defined as line graphs of thetas.  This use of Truemper configurations
is seemingly something deep and general as suggested by Vu\v skovi\'c
in a very complete survey~\cite{vuskovic:truemper} about Truemper
configurations and how they are used (sometimes implicitly) in many
decomposition theorems.

So far, no systematic study of the exclusion of Truemper
configurations has been made. The most interesting question is perhaps
the following.

\begin{question}
  \label{q:noPyWh}
  Are all wheel-free graphs $\chi$-bounded by the same function?
\end{question}

\subsection*{Further reading}

About Truemper configurations, the survey of Vu\v
skovi\'c~\cite{vuskovic:truemper} is the best reading.  About
excluding wheels, see Aboulker, Radovanovi\'c, Trotignon and Vu\v
skovi\'c, \cite{aboulkerRTV:propeller}.  To see how Truemper
configurations appear naturally in the definition of several classes
that generalize chordal graphs, see Aboulker, Charbit, Trotignon and
Vu\v skovi\'c \cite{abChTrVu:moplex}.

Testing whether a graph contains or not some type of Truemper
configuration is a question of interest.  Detecting a theta in some
input graph can be done in time $O(n^{11})$ (see Chudnovsky and
Seymour~\cite{chudnovsky.seymour:theta}) and a pyramid in time
$O(n^9)$ (Chudnovsky, Cornu\'ejols, Liu, Seymour and Vu\v
skovi\'c~\cite{chudnovsky.c.l.s.v:reco}).  Detecting a prism is
NP-complete (Maffray and Trotignon~\cite{maffray.t:reco}).  Detecting
a wheel is NP-complete, even when restricted to bipartite (and
therefore perfect) graphs (Diot, Tavenas and
Trotignon~\cite{diotTaTr:13}).  Detecting a prism or a pyramid can be
done in time $O(n^5)$ (Maffray and Trotignon~\cite{maffray.t:reco});
so detecting a prism in a Berge graph is polynomial, since Berge
graphs contain no pyramids.  Detecting a theta or a pyramid can be
done in time $O(n^7)$ (Maffray, Trotignon and Vu\v
skovi\'c~\cite{maffray.t.v:3pcsquare}).  Detecting a prism or a theta
can be done in time $O(n^{35})$ (Chudnovsky and
Kapadia~\cite{Chudnovsky.Ka:08}).  For similar questions, see
L\'ev\^eque, Lin, Maffray and Trotignon~\cite{leveque.lmt:detect}.

\section{The strategy of the proof}

The main result in~\cite{chudnovsky.r.s.t:spgt} is
Theorem~\ref{th:dec} below, and as we know from the previous sections,
it implies the SPGT.  Its statement is the result of a long sequence
of attempts by many researchers, as explained in the introduction
of~\cite{chudnovsky.r.s.t:spgt} or in~\cite{RousselRT09}.  A slight
variant on what seems now to be the right statement was first
conjectured by Conforti, Cornu\'ejols and Vu\v
skovi\'c~\cite{conforti.c.v:square}.  They proved it in the
square-free case, and some of the arguments that they discovered are
essential in the strategy described below (in particular, the
attachments to prisms, and the use of Truemper configurations).

\begin{theorem}[Chudnovsky, Robertson, Seymour and Thomas 2002]
  \label{th:dec}
  Every Berge graph is basic, or has a 2-join, a complement 2-join, a
  homogeneous pair or a balanced skew partition.
\end{theorem}

The strategy used by Chudnovsky, Robertson, Seymour and Thomas to
prove Theorem~\ref{th:dec} is classical in structural graph theory.
It consists in identifying a `dense' basic class, and a `sparse'
basic class as we explain now.  The `dense' basic class does not
contain the obstruction (here an odd hole or antihole) of course, but
`almost' contains it, so that if a graph $G$ contains an induced
dense subgraph $H$, then any vertex exterior to $H$ must attach in a
very specific way to $H$, either enlarging the basic graph to a bigger
basic graph, or entailing a decomposition.  Therefore, for the sake
of proving the decomposition theorem, it can be assumed that this
particular class of basic graphs is excluded.  Then the process can be
iterated with a new kind of dense basic graphs.  The sparse class is what
remains when all dense substructures are excluded.  Observe that this
method is in some sense wiser than proofs by induction.  Finding the
right induction hypothesis is time consuming because for every
failure, one has to restart from scratch, while a lemma stating that
some dense substructure entails a decomposition is just a true
statement, that can be used forever, even if the strategy of the
proof changes.

For proving Theorem~\ref{th:dec} the sparse class is formed by
bipartite graphs and their complements.  The dense class is more
complicated.  At the beginning of the proof, it is formed by
`sufficiently' connected line graphs, their complements and doubled
graphs.  The simplest line graphs in this context are the odd and even
prisms (that are the line graphs of bipartite thetas).  To understand
why prisms and their generalizations are `dense', the reader can
check as an exercice that a Berge graph formed of a prism and one
vertex not in the prism is either a bigger line graph, or has a
2-join, or some kind of skew partition, namely a star cutset. For this
purpose, it is very convenient to know that Berge graphs have no
pyramids, and what makes this work is that prisms are `close to'
containing pyramids.  Typical cases that should pop out from a proof
attempt are represented in Fig.~\ref{f:prism+v}.  Then the reader
might try to prove a similar statement for the line graph of the
2-subdivision of $K_4$, or to prove similar statements with the vertex
outside of the prism replaced by some path with neighbors in exactly
two paths of the structure for instance.  All this should lead to
variants on 10.1 from~\cite{chudnovsky.r.s.t:spgt}, whose proof is
easy to read since it does not rely on any technical lemmas.  More
about attachments to prisms is explained in Section~\ref{sec:bfp}
below.

\begin{figure}
  \center
  \begin{tabular}{ccc}
    \includegraphics{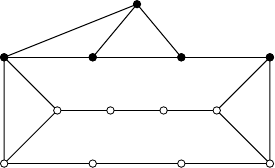}&\rule{4em}{0ex}&
    \includegraphics{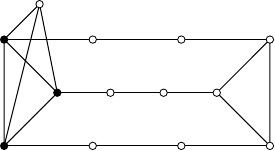}\\
    2-join&&skew partition\\
    \rule{0em}{4ex}&&\\
    \includegraphics{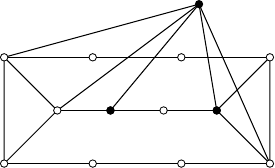}&&
    \includegraphics{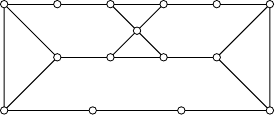}\\
    skew partition&&line graph\\
    \rule{0em}{4ex}&&\\
    \includegraphics{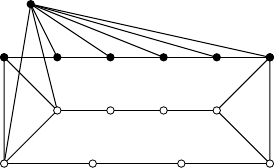}&&
    \includegraphics{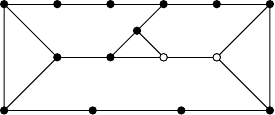}\\
    2-join&&pyramid\\
    \rule{0em}{4ex}&&\\
    \includegraphics{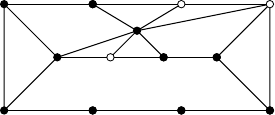}&&
    \includegraphics{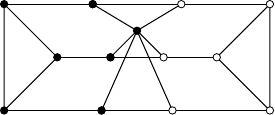}\\
    pyramid&&pyramid\\
    \rule{0em}{3ex}&&\\
  \end{tabular}
  \caption{Various ways to attach a vertex to a prism \label{f:prism+v}} 
\end{figure}

A sequence of about a dozen dense basic graphs is considered:
first, several kinds of line graphs of bipartite subdivisions of
$K_4$, then even prisms, then long prisms (\emph{long} means that at least
one of the paths has length at least~2), then the \emph{double
  diamond} (see Fig~\ref{fig:spo}), then various kinds of wheels (that
are not basic, but that contain skew partitions), and finally
antiholes of length at least 6.  For each of these dense basic
classes, it is proved that containing it leads to being basic or having
some decomposition, and therefore, the Berge graphs handled next may
be assumed not to contain this kind of induced subgraph.  At the end of
this process, so many induced subgraphs are excluded that the graph
under consideration, or its complement, is bipartite.  Needless to
say, identifying this long sequence of `dense' graphs is \emph{tour
  de force}, especially since for each of them, the technicalities are
really involved.  In particular, the self-complementary graphs
$L(K_{3, 3})$ and $L(K_{3,3}\setminus e)$ (see Fig.~\ref{fig:spo}) are
a problem since they are basic in many ways (they are both line graph
and complement of line graphs, and the later is also a doubled graph).
Therefore, there are several ways to describe their structure,
depending on the basic class, and the relevant one is known only from
the rest of the graph.

\begin{figure}
\center
\includegraphics[height=3.5cm]{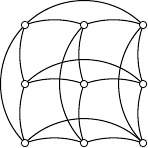}\rule{2em}{0ex}
\includegraphics[height=3cm]{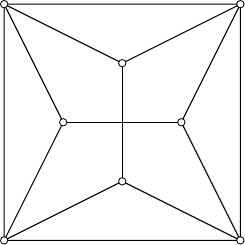}\rule{2em}{0ex}
\includegraphics[height=3cm]{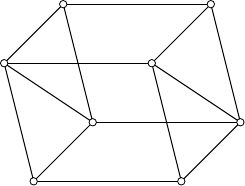}
\caption{$L(K_{3, 3})$, $L(K_{3, 3} \setminus e)$ and the double diamond\label{fig:spo}}
\end{figure}

Despite all the deep technicalities, the objects considered in the
proof of Theorem~\ref{th:dec} are very combinatorial.  This leads to
the following question. 

\begin{question}
  \label{q:polyRead}
  Can the proof of Theorem~\ref{th:dec} be transformed into a polynomial
  time algorithm whose input is any graph $G$ and whose output is either an
  odd hole, an odd antihole, or a partition of the vertices of $G$
  certifying one of the outcomes of Theorem~\ref{th:dec}?
\end{question}

\subsection*{Further reading}

In~\cite{chudnovsky.r.s.t:spgt}, the global strategy of the proof is
well explained at the beginning.  More about the strategy is to be
found in \cite{seymour:how} and \cite{chudvovsky.r.s.t:progress}.  How
structural methods can be used generally for classes closed under
taking induced subgraphs is discussed in Chudnovsky and
Seymour~\cite{chudnovsky.seymour:excluding} and in Vu\v
skovi\'c~\cite{vuskovic:truemper}.  To make a start on
Question~\ref{q:polyRead}, the first step is the detection in
polynomial time of the structures that are used in the proof (line
graphs of a bipartite subdivision of $K_4$, even prism and odd prism).
Apart from wheels (that are NP-complete to detect as mentioned at the
end of Section~\ref{sec:truemper}), they all can be detected in
polynomial time in Berge graph, see Maffray and
Trotignon~\cite{maffray.t:reco}.

\section{The Roussel--Rubio lemma}
\label{sec:RR}

We are now ready to invistigate some technicalities of the proof of
the SPGT.  A lemma due to Roussel and Rubio~\cite{roussel.rubio:01} is
used at many steps in~\cite{chudnovsky.r.s.t:spgt}.  In fact, the
authors of~\cite{chudnovsky.r.s.t:spgt} rediscovered it (in joint work
with Thomassen) and initially named it the \emph{wonderful lemma}
because of its many applications.

The Roussel--Rubio lemma states that, in a sense, any anticonnected
set of vertices of a Berge graph behaves like a single vertex.  How
does a vertex $v$ `behave' in a Berge graph?  If a path of odd
length (at least~3) has both ends adjacent to $v$, then $v$ must have
other neighbors in the path, for otherwise there is an odd hole.  The
lemma states roughly that an anticonnected set $T$ of vertices behaves
similarly: if a path of odd length (at least~3) has both ends complete
to $T$, then at least one internal vertex of the path is also complete
to $T$.  In fact, there are two situations where this statement fails
(outcomes~\ref{oleap} and~\ref{ohop} below),
and for the sake of induction, it is convenient to give similar
properties for a path of even length.  All this results in a more
complicated statement.  A \emph{$T$-edge} is an edge whose ends are
$T$-complete.  When $P = xx'\dots y'y$ is a path of length at least 3,
a \emph{leap} for $P$ is a pair of non-adjacent vertices $u,v$ such that
$N(u) \cap V(P) = \{x, x', y\}$ and $N(v) \cap V(P) = \{x, y', y\}$.
Observe that if $u, v$ is a leap for $P$, then $V(P) \cup \{u, v\}$
induces a prism.  We denote by $P^*$ the interior of a path $P$.

\begin{lemma}[Roussel and Rubio 2001]\label{l:w}
  Let $T$ be an anticonnected set of vertices in a Berge graph $G$.
  If $P$ is a path, vertex-disjoint from $T$ and whose ends are
  $T$-complete, then one the following holds.
  \begin{enumerate}
  \item \label{oeven} $P$ has even length and has an even number of
    $T$-edges;
  \item \label{oodd} $P$ has odd length and has an odd number of
    $T$-edges;
  \item \label{oleap} $P$ has odd length at least $3$ and there is a
    leap for $P$ in $T$;
  \item \label{ohop} $P$ has length $3$ and its two internal vertices
    are the endvertices of an antipath of odd length whose interior is in $T$.
\end{enumerate}
\end{lemma}

\begin{proof}
  We prove the lemma by induction on $|V(P)\cup T|$.  If $P$ has
  length at most~$2$, then we have outcome~\ref{oeven} or~\ref{oodd}.
  So let us assume that $P$ has length at least $3$.  Put $P=xx'\cdots
  y'y$.  Let us suppose that outcomes~\ref{oleap} does not hold for $P$.
  We distinguish between three cases.

  {\noindent\bf Case 1:} There is a $T$-complete vertex in $P^*$.  Let $z$ be
  such a vertex.  By induction, we can apply the lemma to the path
  $xPz$ and $T$.  If $xPz$ has odd length and there is a leap $\{u,
  v\}$ for $xPz$ in $T$, then $(xPz)^* \cup \{u, v, y\}$ induces an odd
  hole.  If $xPz$ has length $3$ and its two internal vertices are the
  endvertices of an odd antipath $Q$ whose interior is in $T$, then
  $Q\cup\{y\}$ induces an odd antihole.  So it must be that the number
  of $T$-edges in $xPz$ and the length of $xPz$ have the same parity.
  The same holds for $zPy$.  So the number of $T$-edges in $P$ and
  the length of $P$ have the same parity, and we have
  outcome~\ref{oeven} or~\ref{oodd}.

  {\noindent\bf Case 2:} $T$ induces a stable set.  We denote by
  $\varepsilon$ the parity of the length of $P$.  Mark the vertices of
  $P$ that have at least one neighbor in $T$.  Call an \emph{interval}
  any subpath of $P$, of length at least~1, whose ends are marked and
  whose internal vertices are not.  Since $x$ and $y$ are marked, the
  edges of $P$ are partitioned by the intervals of $P$.
 
We claim that every interval of $P$ either has even length or has
length~1.  Indeed, suppose there is an interval of odd length, at
least 3, say $P' = x''\dots y''$, named so that $x, x'', y'', y$
appear in this order along $P$.  Let $u$ and $v$ be neighbors of $x''$
and $y''$ in $T$, respectively.  If $x''$ and $y''$ have a common
neighbor $t$ in $T$ then $P'\cup \{t\}$ induces an odd hole.  Hence
$u\neq v$, $x''\neq x$, $y''\neq y$, $v$ is not adjacent to $x''$, and
$u$ is not adjacent to $y''$.  If $x''\neq x'$, then $P'\cup \{ u, x,
v\}$ induces an odd hole. So, $x''=x'$ and similarly, $y''=y'$.
Hence, $\{u, v\}$ is a leap, a contradiction.  This proves our claim.

Hence, the number of intervals of length $1$ in $P$ has parity
$\varepsilon$.  Moreover, we claim that for every interval of
length~1, there is a vertex in $T$ adjacent to both its ends.  Indeed,
suppose that there is an interval $x''y''$ such that $x''$ and $y''$
do not have a common neighbor in $T$.  Let $u$ be a neighbor of $x''$
in $T$, and let $v$ be a neighbor of $y''$ with $u \neq v$,
$uy''\not\in E(G)$, and $vx'' \not\in E(G)$.  Note that $x\neq x''$
and $y\neq y''$.  If $x''\neq x'$, then $\{ u, x, v, x'', y''\}$
induces an odd hole.  So, $x''=x'$ and similarly $y''=y'$.  Now $\{u,
v\}$ is a leap, a contradiction.

For every $v\in T$, denote by $f(v)$ the set of all $\{v\}$-complete
edges of $P$.  Let $v_1, \dots, v_n$ be the elements of $T$. We know
that $|f(v_1) \cup \dots \cup f(v_n)|$  has parity $\varepsilon$,
since, from the previous paragraph, it is equal to the number of the
intervals of length~1. Moreover, by the sieve formula we have:

  \begin{eqnarray}
    |f(v_1)   \cup \dots \cup f(v_n)| & = & \sum_{i} |f(v_i)| \nonumber \\
    & & -  \sum_{i\neq j} |f(v_i) \cap f(v_j)| \nonumber \\
    & & \vdots \nonumber
  \end{eqnarray}
  \begin{eqnarray}
    & & + (-1)^{(k+1)} \sum_{I\subset \{1, \dots, n\}, |I|=k} |\cap_{i\in I} f(v_i)| \nonumber \\
    & & \vdots \nonumber \\
    & & + (-1)^{(n+1)} |f(v_1) \cap \dots \cap f(v_n)| \nonumber
  \end{eqnarray}

  By the induction hypothesis, we know that if $S\subsetneq T$, then
  the number of $S$-complete edges in $P$ has parity $\varepsilon$.
  Hence if $I\subsetneq \{1, \dots, n\}$, then $|\cap_{i\in I}
  f(v_i)|$ has parity $\varepsilon$.  Thus, we can rewrite the above
  equality modulo 2 as:
  
  \[
  |f(v_1) \cup  \dots \cup f(v_n)|= (2^n-2)  +
  (-1)^{(n+1)} |f(v_1)  \cap \dots \cap f(v_n)|
  \]
    
  Since $|f(v_1) \cup \dots \cup f(v_n)|$ has parity $\varepsilon$, it
  follows that $|f(v_1) \cap \dots \cap f(v_n)|$ has parity
  $\varepsilon$, meaning that the number of $T$-edges in $P$ has
  parity $\varepsilon$.  It follows that one of \ref{oeven} of
  \ref{oodd} holds.

  {\noindent\bf Case 3:} We are neither in Case~1 nor in Case~2 (so
  $T$ is not a stable set and there is no $T$-complete vertex in
  $P^*$).  Let $Q=u\cdots v$ be a longest path of $\overline{G}[T]$.
  So $Q$ has length at least $2$ (since $T$ is not a stable set), and
  $T\setminus \{u\}$ and $T \setminus \{v\}$ are anticonnected sets.
  By the induction hypothesis, we know that $P$ has an odd number of
  $T\setminus \{u\}$-edges and an odd number of $T \setminus
  \{v\}$-edges.  Note that a $T \setminus \{u\}$-edge and a $T
  \setminus \{v\}$-edge have no common vertex, for otherwise there
  would be a $T$-complete vertex in $P^*$.  In particular all $T
  \setminus \{u\}$-edges and $T\setminus\{v\}$-edges are different.

  Suppose that $Q$ has even length.  Let $x_u x'_u$ be a $T\setminus
  \{u\}$-edge of $P$ and $y'_v y_v$ be a $T\setminus \{v\}$-edge of
  $P$ such that, without loss of generality, $x, x_u, x'_u, y'_v, y_v,
  y$ appear in this order on $P$.  If $x'_u$ is non-adjacent to $y'_v$
  then $\{x'_u, y'_v\} \cup Q$ induces an odd antihole.  If $x\neq
  x_u$ then $\{x_u, y'_v\} \cup Q$ induces an odd antihole.  If $y_v
  \neq y$ then $\{x'_u, y_v\} \cup Q$ induces an odd antihole.  It
  follows that $P = x_u x'_u y'_v y_v$, but then $P\cup Q$ induces an
  odd antihole.  Thus $Q$ has odd length (at least $3$).

  Suppose that $T\setminus\{u,v\}$ is not anticonnected.  Since $T
  \setminus \{u\}$ and $T\setminus\{v\}$ are anticonnected, there
  exists a vertex $w$ in an anticomponent of $G[T \setminus \{u,v\}]$
  that does not contain $Q^*$ and such that $w$ is adjacent in
  $\overline{G}$ to at least one of $u,v$; but then $Q\cup\{w\}$
  induces in $\overline{G}[T]$ either a chordless path longer than $Q$
  or an odd hole, a contradiction.  So $T\setminus\{u,v\}$ is
  anticonnected.

  Now we know that there is an odd number of $T\setminus \{u,
  v\}$-edges in $P$ (by the induction hypothesis).  Recall that $P$
  has an odd number of $T\setminus\{u\}$-edges, an odd number of
  $T\setminus\{v\}$-edges, and that these are different, so these
  account for an even number of $T\setminus \{u, v\}$-edges; thus $P$
  has at least one $T\setminus \{u, v\}$-edge $x'' y''$ that is
  neither a $T\setminus \{u\}$-edge nor a $T \setminus \{v\}$-edge.
  We may assume that $x,x'',y'',y$ appear in this order along $P$ and
  that $y''\in P^*$.  So $y''$ is non-adjacent to one of $u, v$, say
  $v$.  Then $y''$ is adjacent to $u$, for otherwise $Q \cup \{y''\}$
  would induce an odd antihole.  Then $x''$ is non-adjacent to $u$,
  for otherwise $x''y''$ would be a $T\setminus \{v\}$-edge.  Then
  $x''$ is adjacent to $v$, for otherwise $Q\cup \{x''\}$ would induce
  an odd antihole.  Then $x''=x'$ for otherwise $Q\cup\{x'', y'', x\}$
  would induce an odd antihole, and similarly $y''=y'$.  So
  $P=xx''y''y$ and $Q\cup\{x'', y''\}$ is a chordless odd path of
  $\overline{G}$, and we have outcome~\ref{ohop}.
\end{proof}

It is not easy to see how useful  Lemma~\ref{l:w} is, so let us give
now a simple application, that is 3.1
from~\cite{chudnovsky.r.s.t:spgt}.

\begin{lemma}
  In a Berge graph, if a hole $C$ and an antihole $D$ both have
  length at least~8, then $|V(C) \cap V(D)| \leq 3$.   
\end{lemma}

\begin{proof}
  It is easy to check that $P_4$ is the only graph $H$ on at least four
  vertices such that both $H$ and $\overline{H}$ are subgraphs of some
  path.  It follows that if $|V(C) \cap V(D)| \geq 4$, then
  $V(C) \cap V(D)$ induces a $P_4$, say $abcd$.  So, there  is a path
  $P$ from $a$ to $d$ in $G$,  such that $abcdPa$ is the hole $C$; and
  there is an antipath $Q$ from $b$ to $c$ such that $bQcadb$ is the
  antihole $D$.  Note that $P$ and $Q$ are both of odd length, at
  least~5. 

  The ends of $P$ are $Q^*$-complete.  If $P^*$ contains a
  $Q^*$-complete vertex $v$, then $\{v\} \cup V(Q)$ induces an odd
  antihole.  Therefore, by Lemma~\ref{l:w} $Q*$ contains a leap $\{u,
  v\}$ for $P$, so some path $P'$ from $u$ to $v$ has the same
  interior as $P$.  Observe that $u$ and $v$ are consecutive along
  $Q$, and because of the length of $Q$, one of $b,c$ (say $b$) is
  complete to $\{u, v\}$.  It follows that $V(P') \cup \{b\}$ induces
  an odd hole.
\end{proof}

We now give a corrolary of~\ref{l:w} that is used constantly
in~\cite{chudnovsky.r.s.t:spgt} (where it is called 2.2).

\begin{corollary}
  \label{c:rr2}
  Let $T$ be an anticonnected set of vertices in a Berge graph $G$.
  If $P$ is a path with odd length at least~3, vertex-disjoint from
  $T$, whose ends are $T$-complete and such that no internal vertex of
  $P$ is $T$-complete, then every $T$-complete vertex has a neighbor
  in $P^*$.
\end{corollary}

\begin{proof}
  Let $v$ be a $T$-complete vertex, and suppose that $v$ has no
  neighbor in $P^*$ (so $v$ is not in $V(P)\cup T$).  Apply
  Lemma~\ref{l:w} to $T$ and $P$.  If outcome~\ref{oleap} of
  Lemma~\ref{l:w} holds, then a path of odd length with same interior
  as $P$ joins the members of the leap.  Together with $v$, it forms
  an odd hole, a contradiction.  If outcome~\ref{ohop} of
  Lemma~\ref{l:w} holds, then the antipath of odd length can be
  completed to an odd antihole through $v$, a contradiction.
\end{proof}

\subsection*{Further reading}

Another proof of the Roussel--Rubio lemma is given in Maffray and
Trotignon~\cite{nicolas:artemis}, where some applications are also
presented.  The proof given here in the case when $T$ is a stable set
is due to Kapoor, Vu\v skovi\'c and Zambelli,
see~\cite{nicolas.kristina:RR} where several very simple applications
are presented.

\section{Book from the Proof}
\label{sec:bfp}

The goal of this section is to present a self-contained lemma
of~\cite{chudnovsky.r.s.t:spgt} in order to give some taste of the
technicalities.  Let us replace this lemma in its context by stating
informally how line graphs are handled. 

A line graph $K$ of a bipartite graph is formed of a bunch of cliques,
and a bunch of paths linking them.  If $K$ is contained in a Berge
graph $G$, a vertex $v\in V(G\setminus K)$ is \emph{major} if it has
many (here at least~2) neighbors in each of the cliques, and
\emph{minor} if it has neighbors in at most one of the paths, or in at
most one of the cliques.  An important result is that every vertex is
major or minor, or allows one to obtain a larger line graph.  This is not
fully true (Fig.~\ref{f:prism+v} contains counter-examples), but
vertices that are neither major nor minor can be considered as part of
some of the paths in what is called a generalized line graph.
Therefore, when the generalized line graphs are properly defined, it
is true that every vertex is major or minor w.r.t.\ a maximal
generalized line graph.  A next step is to prove that connected
(sometimes anticonnected) components of vertices of $G\setminus K$
behave in fact as a vertex.  A reader wanting to know more details for
components can read the proof of 10.1 in~\cite{chudnovsky.r.s.t:spgt}
and we deal with anticomponents below.

If there are no major vertices, the graph is formed of the line graph,
and possibly a bunch of minor components.  If some of these components
attach to a clique, then there is a balanced skew partition, and if some
attach to a path, there is a 2-join.  If there are major vertices,
then it can be proved that an anticomponent of these is formed of
vertices that all major \emph{in the same way}, meaning that they all
attach to exactly the same vertices of the cliques.  Therefore, an
anticomponent of major vertices is complete to its attachment, and
forms a skew cutset separating parts of the line graph.  The next
lemma (it is a variant of 7.3 from~\cite{chudnovsky.r.s.t:spgt})
proves a statement of this form.  It illustrates another breakthrough
made in~\cite{chudnovsky.r.s.t:spgt}: how the Roussel-Rubio lemma can
be used to find skew partitions.

\begin{lemma}
  In a Berge graph $G$, let $K$ be a prism with triangles $a_1a_2a_3$
  and $b_1b_2b_3$ and paths $P_i = a_i\dots b_i$, $i=1, 2, 3$. Suppose
  that for $i=1, 2, 3$, $P_i$ has length at least~2.  Let $Y$ be an
  anticonnected set of vertices such that each of them have at least
  two neighbors in $\{a_1, a_2, a_3\}$ and at least two neighbors in
  $\{b_1, b_2, b_3\}$ (so $Y$ is disjoint from $K$).  Then at least
  two members of $\{a_1, a_2, a_3\}$ and at least two members of
  $\{b_1, b_2, b_3\}$ are $Y$-complete.
\end{lemma}

\begin{proof}
  Suppose not; then there is an antipath with interior in $Y$ joining
  two vertices both in $\{a_1, a_2, a_3\}$ or both in $\{b_1, b_2,
  b_3\}$.  Let $Q$ be a shortest such antipath.  Suppose up to
  symmetry that $Q$ is from $a_1$ to $a_2$.  Every vertex in $Y$ is
  adjacent to either $a_1$ or $a_2$, so $Q$ has length at least~3.
  From the minimality of $Q$, $a_3$ is $Q^*$-complete, and so is at
  least one of $\{b_1, b_2, b_3\}$, say $b_i$.  Since $Q$ can be
  completed to an antihole via $a_1b_ia_2$, it follows that $Q$ has
  even length, therefore at least 4.  We set $Q = a_1 q_1 \dots q_n
  a_2$.  Let $a'_i$ be the neighbor of $a_i$ in $P_i$, $i=1, 2$ and $P
  = a'_1 P_1 b_1 b_2 P_2 a'_2$.

  \begin{claim}
    \label{c:xe}
    At least one internal vertex of $P$ is $(Q^*\setminus q_1)$-complete
    and at least one internal vertex of $P$ is $(Q^*\setminus
    q_n)$-complete.
  \end{claim}
  
  \begin{proofclaim}
    If one of $b_1$ or $b_2$ is $Q^*$-complete, the claim is obviously
    true. Otherwise, none of $b_1, b_2$ is $Q^*$-complete so there
    exists a antipath from $b_1$ to $b_2$ whose interior in $Q^*$, and
    from the minimality of $Q$ this antipath has the same interior as
    $Q$. It follows that one of $b_1, b_2$ is complete to
    $Q^*\setminus q_1$ and the other one is complete to $Q^*\setminus
    q_n$.
  \end{proofclaim}

  \begin{claim}
    \label{c:C}
    If an internal vertex of $P$ is $(Q^*\setminus q_1)$-complete or
    $(Q^* \setminus q_n)$-complete, then it is $Q^*$-complete.  If $a'_1$
    is $(Q^*\setminus q_1)$-complete, then it is $Q^*$-complete.
  \end{claim}
  
  \begin{proofclaim}
    If an internal vertex $v$ of $P$ is $(Q^*\setminus q_n)$-complete,
    then it is $Q^*$-complete for otherwise, $v q_n Q a_1 v$ is an odd
    antihole.  If $v$ is $a'_1$ or is an internal vertex of $P$, and $v$  is
    $(Q^*\setminus q_1)$-complete, then $v$ is $Q^*$-complete, for
    otherwise, $v q_1 Q a_2 v$ is an odd antihole.
  \end{proofclaim}

  If both $a'_1, a'_2$ are $Q^*$-complete, then $Q$ can be completed
  to an odd antihole via $a_1 a'_2 a'_1 a_2$, a contradiction.  So, up
  to symmetry, we suppose from here on that $a'_1$ is not
  $Q^*$-complete.  It follows by~(\ref{c:C}) that $a'_1$ is not $Q^*
  \setminus q_1$ complete.

  Call $x$ the $Q \setminus q_1$ complete vertex of $P$ closest to
  $a'_1$ along $P$.  By~(\ref{c:xe}), $x$ exists and is an internal
  vertex of $P$. By~(\ref{c:C}), $x$ is in fact $Q^*$-complete.

  If $a'_1 a_1 P x$ has odd length, then by Corollary~\ref{c:rr2} applied to
  $a_1 a'_1 Px$ and $Q^* \setminus q_1$, there is a contradiction
  because $a_3$ is $(Q^* \setminus q_1)$-complete and has no neighbor in
  the interior $a_1 P x$.  So $a_1 P x$ has even length.  It follows
  that $x P a'_2 a_2$ has odd length.  By Corollary~\ref{c:rr2} applied to
  $xP a'_2 a_2$ and $Q^* \setminus q_n$, and because of $a_3$, there
  must be an internal vertex $v$ of $x P a'_2 a_2$ that is $(Q^*
  \setminus q_n)$-complete.  By~(\ref{c:C}), $v$ is in fact
  $Q^*$-complete.  But then, $a_3a_1 a'_1 P x$ is odd, its ends are
  $Q^*$-complete, but none of its internal vertex is $Q^*$-complete.
  This contradicts Corollary~\ref{c:rr2} because of $v$.
\end{proof}

The lemma above shows how the Roussel--Rubio lemma `generates'
antiholes to provide contradictions.  Maybe it can be used to
invistigate the structure of several subclasses of Berge graphs, for
instance the next one.

\begin{question}
  \label{q:noAh}
  Is there a structural characterization of Berge graphs with no
  antihole of length at least~6?
\end{question}

\subsection*{Further reading}

About technicalities, the best reading is of course Chudnovsky,
Robertson, Seymour and Thomas~\cite{chudnovsky.r.s.t:spgt}.  The paper
is very well organized: the first four sections are devoted to
technical lemmas that are used extensively in the sequel.  But as a
result of this wise organization, it is very hard to extract a
meaningful part: it is difficult to feel how useful  the
technical lemmas are without knowing the sequel, and the sequel is
difficult to understand without mastering the technicalities of the
lemmas.  A reader seeking the easiest parts
from~\cite{chudnovsky.r.s.t:spgt} should try 15.1, Section 15 or
Section 16.  Section 10 is devoted to the even prism and is a
self-contained chunk that has fewer technicalities than the rest of the
paper (but still relies a lot on technical lemmas of the previous
sections).

A reader might want to read self-contained papers using the same kind
of technicalities as \cite{chudnovsky.r.s.t:spgt} in simpler
situations.  For attachement to line graphs, a possible reading is
L\'ev\^eque, Maffray and Trotignon~\cite{nicolas:isk4} (where
Question~\ref{q:noPyWh} is partially answered by the way).  About how
attachments to a wheel can be used to decompose a class, Radovanovi\'c
and Vu\v skovi\'c~\cite{radovanovicV:theta} is a good reading.

\section{Recognizing perfect graphs}

From here on, we investigate works on perfect graphs that were done
after the proof of the SPGT.  Soon after the proof of the SPGT,
another open question was solved.  Chudnovsky, Cornu\'ejols, Liu,
Seymour and Vu\v skovi\'c~\cite{chudnovsky.c.l.s.v:reco} found a
polynomial time algorithm that decides whether an input graph is
Berge.  Observe that this algorithm is independent from the SPGT.  It
takes any graph $G$ as an input and outputs in time $O(n^9)$ an odd
hole-or-antihole of $G$ (if any).  We give here a brief outline (that
is a copy from the introduction of~\cite{chudnovsky.c.l.s.v:reco}).
In what follows, $G$ is the input graph of the algorithm. 
 
\begin{quote}
  [\dots] we would like to decide either that $G$ is not Berge, or
  that $G$ contains no odd hole. (To test Bergeness, we just run this
  algorithm on $G$ and then again on the complement of $G$.)  If there
  is an odd hole in $G$, then there is a shortest one, say $C$. A
  vertex of the remainder of $G$ is $C$-major if its set of neighbours
  in $C$ is not a subset of the vertex set of any 3-vertex path of
  $C$; and $C$ is \emph{clean} (in $G$) if there are no $C$-major
  vertices in $G$.  If there happens to be a clean shortest odd hole
  in $G$, then it stands out and can be detected relatively easily;
  and that essentially is the first step of our algorithm, a routine
  to test whether there is a clean shortest odd hole. The remainder of
  the algorithm consists of reducing the general problem to the
  `clean' case that was just handled.  If $C$ is a shortest odd hole
  in $G$, let us say a subset $X$ of $V(G)$ is a \emph{cleaner} for
  $C$ if $X \cap V(C)=\emptyset$ and every $C$-major vertex belongs to
  $X$. Thus if $X$ is a cleaner for $C$ then $C$ is a clean hole in
  $G\setminus X$.  The idea of the remainder of the algorithm is to
  generate polynomially many subsets of $V(G)$, such that if there is
  a shortest odd hole $C$ in $G$, then one of the subsets will be a
  cleaner for $C$. If we can do that, then we delete each of these
  subsets in turn, thereby generating polynomially many induced
  subgraphs; and we know that there is an odd hole in $G$ if and only
  if in one of these subgraphs there is a clean shortest odd
  hole. Thus we can decide whether $G$ has an odd hole by testing
  whether any of these subgraphs has a clean shortest odd hole.
\end{quote}

\begin{theorem}[Chudnovsky, Cornu\'ejols, Liu,
  Seymour and Vu\v skovi\'c 2002]
  There exists an algorithm that decides whether an input
  graph is Berge in time  $O(n^9)$.  
\end{theorem}

Let us now explain briefly the two steps from the sketch above.  The
second step relies on a powerful technique discovered by Conforti and
Rao~\cite{ConfortiR:93}, called \emph{cleaning}, which consists in
`guessing' a cleaner.  Here, the way to guess the cleaner comes from
Roussel--Rubio like lemmas, saying that the set $X$ of vertices
vertices to be `guessed', that are all major vertices of some possible
smallest odd hole, are all common neighbors of some set $S$ (of size
bounded by constant) of neighbors.  Therefore, one can enumerate all
possible $S$'s by brute-force, and for each of them nominate $X$ as
the set of $S$-complete vertices.

The first step relies on the \emph{shortest path detector}, a method
designed by Chudnovsky and Seymour~\cite{chudnovsky.c.l.s.v:reco} that
is used twice in the algorithm: once to detect a clean odd hole, and
once to detect a pyramid.  Detecting a pyramid is needed in a
preprocessing step that we omitted to explain in the sketch above.  It
consists in the detection of several substructures certifying that the
graph is not Berge, and one of them is the pyramid (that contains an
odd hole).

To explain the shortest path detector, we give here a simpler
algorithm that detects prisms in graphs with no pyramids (recall that
in general graphs, the problem is NP-complete).  This algorithm is not
published because a faster one exists, see~\cite{maffray.t:reco}.  A
shortest path detector always relies on a lemma stating in the
smallest substructure of the kind that we are looking for, a path
linking two particular vertices of the substructure can be replaced by
any shortest path.

\begin{lemma}
  \label{l:kpri}
  Let $G$ be a graph with no pyramid.  Let $K$ be a smallest prism in
  $G$.  Suppose that $K$ is formed by paths $P_1, P_2, P_3$,
  with triangles $\{a_1, a_2, a_3\}$ and $\{b_1, b_2, b_3\}$, so that,
  for $i=1, 2, 3$, path $P_i$ is from $a_i$ to $b_i$.  Then:
  
  If $R_i$ is any shortest path from $a_i$ to $b_i$ whose interior
  vertices are not adjacent to $a_{i+1}$, $a_{i+2}$, $b_{i+1}$ or
  $b_{i+2}$, then $R_{i}, P_{i+1}, P_{i+2}$ form a prism on $|V(K)|$
  vertices in $G$, with triangles $\{a_1, a_2, a_3\}$ and $\{b_1, b_2,
  b_3\}$ (the addition of subscripts is taken modulo 3).
\end{lemma}

\begin{proof}
  Suppose that the lemma fails for say $i=1$.  So, some interior
  vertex of $R$ has neighbors in the interior of $P_2$ or $P_3$.  Let
  $x$ be such a vertex, closest to $a_1$ along $R$.  Let $a'_2$
  (resp.\ $a'_3$) be the neighbor of $a_2$ (resp.\ $a_3$) along $P_2$
  (resp.\ $P_3$).  Let $Q = a'_2 P_2  b_2  b_3  P_3 
  a'_3$.  Let $y$ (resp.\ $z$) be the neighbor of $x$ closest to
  $a'_2$ (resp.\ $a'_3$) along $Q$.

  If $y=z$ then $y x  R  a_1$, $y  Q  a'_2  a_2$ and
  $y  Q  a'_3  a_3$ form a pyramid, a contradiction.  If
  $y\neq z$ and $yz\notin E(G)$ then $x  R  a_1$, $x y  Q
   a'_2  a_2$ and $x  z  Q  a'_3  a_3$ form a
  pyramid, a contradiction.  If $yz\in E(G)$ then $x  R  a_1$, $y  Q
   a'_2  a_2$ and $z  Q  a'_3  a_3$ form a prism on
  less vertices  than $K$, a contradiction. 
\end{proof}

Now detecting a prism in a graph with no pyramid can be performed as
follows.  For all 6-tuples $(a_1, a_2, a_3, b_1, b_2, b_3)$ compute
three shortest paths $R_i$ in $G\setminus ((N[a_{i+1}] \cup N[a_{i+2}] \cup
N[b_{i+1}] \cup N[b_{i+2}]) \setminus \{a_i, b_i\})$ from $a_i$ to $b_i$.
Check whether $R_1, R_2, R_3$ form a prism, and if so output it.  If
no triple of paths forms a prism, output that the graph contains no
prism.  If the algorithm outputs a prism, this is obviously a correct
answer: the graph contains a prism.  Suppose conversely that the graph
contains a prism.  Then it contains a smallest prism with triangles
$\{a_1, a_2, a_3\}$ and $\{b_1, b_2, b_3\}$.  At some step, the
algorithm will check this 6-tuple (unless a prism is discovered before,
but then the correctness is proved anyway).  By three applications of
Lemma~\ref{l:kpri}, we see that the three paths $R_1, R_2, R_3$ form a
prism that is output.  All this take time $O(n^8)$.

The following is still open.

\begin{question}
Is there a polynomial time algorithm to decide whether an input graph
has an odd hole? 
\end{question}

\subsection*{Further reading}

In~\cite{chudnovsky.c.l.s.v:reco}, another algorithm for recognizing
Berge graphs is given. It relies on decompositions and uses
Theorem~\ref{th:double}.  Detecting odd holes can be solved in
polynomial time under the assumption that the largest size of a clique
in the input graph is bounded by some constant.  This is proved by
Conforti, Cornu\'ejols, Liu, Vu\v skovi\'c and
Zambelli~\cite{conforti.c.l.v.z:oddHoleBounded} that is a good reading
for understanding the main ideas of the recognition of Berge
graphs.  A combination of cleaning and shortest path detector is used
by Chudnovsky, Seymour and Trotignon~\cite{chudnovsky.s.T:net} to
decide in polynomial time whether a graph contains a subdivision of
the net (the \emph{net} is the graph obtained from a triangle by adding
a pendant edge at each vertex).

\section{Berge trigraphs}
\label{sec:trigraphs}

An obvious question about Theorem~\ref{th:dec} is whether it is best
possible.  Is each outcome necessary?  Are there outcomes that can be
made more precise?  A careful reader of~\cite{chudnovsky.r.s.t:spgt} might
notice that the outcome `homogeneous pair' is obtained
only once in the whole proof, and might therefore wonder whether it is
really necessary.  Chudnovsky proved it is not.  Let us explain how.

A way to proceed is to consider a smallest Berge graph $G$ such that
the homogeneous pair is the only outcome satisfied by $G$ in
Theorem~\ref{th:dec}, and to look for a contradiction.  A natural idea
is then to `contract' a homogeneous pair $(A, B)$ in order to find a
smaller Berge graph $G'$.  Then, apply Theorem~\ref{th:dec} to $G'$,
and prove that any outcome of Theorem~\ref{th:dec} in $G'$ yields an
outcome in $G$ (because $G'$ and $G$ are very similar).  From the
initial assumption, $G'$ therefore satisfies no outcome of
Theorem~\ref{th:dec}, and this provides a contradiction.  We call this
method \emph{bootstrap}, because it improves structural theorems
somehow for free.  The natural way to `contract' a homogeneous pair
$(A, B)$ of $G$ is to replace $A$ (resp.\ $B$) by a vertex $a$ (resp.\
$b$) complete to $C(A)$ (resp.\ $C(B)$) and anticomplete to $\overline{C}(A)$
(resp.\ $\overline{C}(B)$).  The bootstrap method is hard to implement in this
context.  The problem is with $ab$: should it be an edge or an antiedge
of $G'$?  Both choices lead to difficult technicalities: if $ab$ is
chosen to be an antiedge, it could be that that some skew cutset
separates $a$ from $b$ in $G'$, while $A$ and $B$ are linked in $G$.
If $a$ is chosen to be adjacent to $b$, it could be that the same
phenomenon happens in the complement.  No neat example of these bad
phenomena can be given, because as we will see, it is true that the
outcome `homogeneous pair' is not necessary in Theorem~\ref{th:dec}.
But any attempt of proof will face this issue and is likely to fail.

The idea of Chudnovsky is to leave undecided the adjacency between $a$
and $b$ in $G'$.  To this purpose, she defines \emph{trigraphs} as
graphs with edges, antiedges, and a third kind of adjacency:
\emph{switchable pair}.  For every pair of distinct vertices $x$ and
$y$, $xy$ is an edge, an antiedge or a switchable pair.  A
\emph{realization} $G$ of a trigraph $T$ is any graph on $V(T)$ such
that all edges (resp.\ antiedges) of $T$ are edges (resp.\ antiedges)
of $G$ (so, every switchable pair is transformed into an edge or an
antiedge).  A trigraph is \emph{Berge} if every realization is Berge.
We do not give the long list of definitions that translates naturally
the vocabulary of graphs to trigraphs.  The key point is in the
definitions of the decompositions: all the `important' edges in
definitions of decompositions are not allowed to be switchable pairs.
For instance, if $(X, Y)$ is a 2-join of a trigraph $T$, no switchable
pair of $T$ is from $X$ to $Y$.  If $(X, Y)$ is a skew partition of a
trigraph $T$, it must be that $X$ is partitioned into $X_1$ and $X_2$
such that no edge and no switchable pair exists between $X_1$ and
$X_2$, and it must be that $Y$ is partitioned into $Y_1$ and $Y_2$ such
that no antiedge and no switchable pair exists between $Y_1$ and
$Y_2$.  It is easy to guess how useful this requirement is: for
instance, when building $G'$ from $G$ as in the paragraph above, $ab$
is defined to be a switchable pair.  So, the problem that we mentioned
with the skew cutset separating $a$ from $b$ does not exists anymore.

Of course, the slight problem with this notion of trigraph is that the
proof of Theorem~\ref{th:dec} has to be done again from the beginning
for Berge trigraphs.  Chudnovsky proved several decomposition theorems
similar to Theorem~\ref{th:dec} for Berge trigraphs.  The proof of the
main one is self-contained and runs along more than 200 pages.  We do
not give the precise statements here, it would be pointless since we
do not give the precise definitions of decompositions and basic
classes of trigraph.  The precise statements of the theorems are
in~\cite{chudnovsky:trigraphs}, where only parts of the proofs are
given.  The complete proof is in~\cite{chudnovsky:these}.  With
trigraphs, the bootstrap method works smoothly and yields the
following (that we translate back to graphs, since a graph is a
particular trigraph).

\begin{theorem}[Chudnovsky 2003]
\label{th:Mdec}
Every Berge graph is basic, or has a 2-join, a complement 2-join or a
balanced skew partition.
\end{theorem}

An interesting feature of trigraphs is that additional conditions can
be added.  For instance, a \emph{monogamous} trigraph is a trigraph
such that every vertex is member of at most one switchable pair.
Monogamous trigraphs are very convenient to handle the interactions of
odd 2-joins and homogenous pairs.  When a monogamous trigraph $T$ has
a homogeneous pair $(A, B)$, a smaller trigraph $T'$ can be
constructed as follows: replace $A$ (resp.\ $B$) by a vertex $a$
(resp.\ $b$) strongly complete to $C(A)$ (resp.\ $C(B)$) and strongly
anticomplete to $\overline{C}(A)$ (resp.\ $\overline{C}(B)$) (here
\emph{strongly} means that only real edges are used, not switchable
pairs), and link $a$ to $b$ by a switchable pair.  When a monogamous
trigraph $T$ has an odd 2-join $(X_1,X_2)$ with sets $A_1, B_1, C_1,
A_2, B_2, C_2$ as in the definition, a smaller trigraph $T_2$ can be
constructed as follows: delete $C_1$ replace $A_1$ (resp.\ $B_1$) by a
vertex $a_1$ (resp.\ $b_1$) strongly complete to $A_2$ (resp.\ $B_2$)
and strongly anticomplete to $B_2 \cup C_2$ (resp.\ $A_2\cup C_2$),
and link $a_1$ to $b_1$ by a switchable pair.  When applied
iteratively, this way of constructing blocks of decompositions
preserves the property of being monogamous, and shows that all the
decompositions used in the process do not cross.  A slightly more
general notion of trigraph is defined in~\cite{chudnovskyTTV:noBsp} to
handle interactions between 2-joins, complement 2-joins and
homogeneous pairs.

Another potential improvement of Theorem~\ref{th:dec} is about the
technical requirements in the definition of a 2-join.  Some authors
consider only \emph{non-path 2-joins} (already defined
page~\pageref{pth2join}).  In some applications, it is essential to use
non-path 2-joins, because one needs to replace one side of the 2-join
by a long path to recurse (in proof by induction, or in algorithms),
and this obviously fails if the side is already a long path.
Trotignon~\cite{nicolas:bsp} investigated this question and obtained
the following result (with the bootstrap method for graphs starting from
Theorem~\ref{th:Mdec}).  \emph{Path cobipartite} graphs and \emph{path
double-split graphs} are Berge graphs obtained by subdividing edges in
complement of bipartite graphs and in double split graphs respectively
(a more precise definition can be given, but this one is enough here).
Observe that subdividing an edge in any graph creates a path 2-join.

\begin{theorem}[Trotignon 2008]
  \label{th:n}
  If $G$ is a Berge graph, then $G$ is basic, or one of $G$,
  $\overline{G}$ is a path-cobipartite graph, or one of $G$,
  $\overline{G}$ is a path-double split graph, or one of $G$,
  $\overline{G}$ has a non-path 2-join, or $G$ has a balanced skew
  partition, or one of $G$, $\overline{G}$ has a homogeneous pair and
  a path 2-join.
\end{theorem}

Observe that this theorem shows that homogeneous pairs are not
necessary to decompose Berge graphs, that path 2-joins are not necessary to
decompose Berge graphs, but does not show that one can get rid of both
outcomes.  We now give examples showing that every outcome of
Theorem~\ref{th:n} is needed.  In Fig.~\ref{fig:contrex1}
and~\ref{fig:contrex2}, a path cobipartite graph and a path
double-split graph are represented.  They both have a path 2-join,
showing that to get rid of path 2-joins, the two new basic classes are
really needed.  The graph represented in Fig.~\ref{fig:contrex3} is
interesting because it is decomposable only by a path 2-join or by a
homogeneous pair, showing that it is impossible to get rid of both
outcomes.  Obtaining graphs uniquely decomposable by a 2-join (we mean
that that they are not basic and that the $2$-join is the only way to
decompose them) is easy by the following recipe.  Consider a Berge
graph $G_i$ ($i=1, 2$) that contains a path $P_i$ of length~3 from
$a_i$ to $b_i$ and such that $A_i=N(a_i)$ and $B_i = N(b_i)$ are
disjoint.  Now take the disjoint union of $G_1$ and $G_2$, and add all
possible edges between $A_1$ and $A_2$ and between $B_1$ and $B_2$.
The resulting graph obviously has a 2-join, and if $G_1$ and $G_2$ are
sufficiently general, this is the only decomposition.  This recipe can
be applied to the graphs from Fig.~\ref{fig:contrex1}
and~\ref{fig:contrex2} for instance.  On Fig.~\ref{wkbgsf}, a Berge
graph uniquely decomposable by a balanced skew partition is
represented.

\begin{figure}[p]
  \begin{center}
    \includegraphics{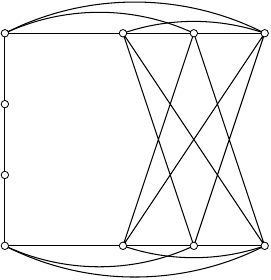}
    \caption{A path-cobipartite graph\label{fig:contrex1}}
  \end{center}
\end{figure}

\begin{figure}[p]
  \begin{center}
    \includegraphics{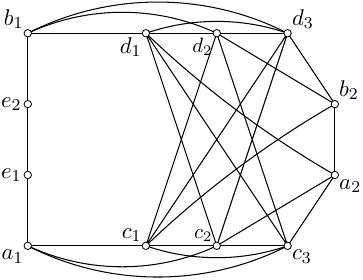}
    \caption{A path-double split graph\label{fig:contrex2}}
  \end{center}
\end{figure}

\begin{figure}[p]
  \begin{center}
    \includegraphics{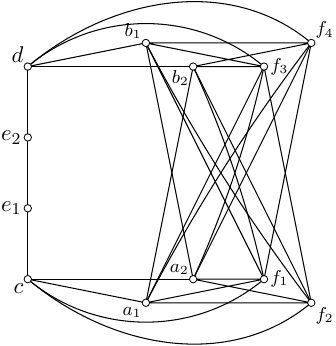}
    \caption{A graph that has a homogeneous pair $(\{a_1, a_2\},$
      $\{b_1, b_2\})$ and a path 2-join\label{fig:contrex3}}
  \end{center}
\end{figure}

Not much work has been devoted to the algorithmic aspects of trigraphs.
The following is still open.

\begin{question}
  What is the complexity of recognizing Berge trigraphs?
\end{question}

\subsection*{Further reading}

It seems now that trigraphs are an important general tool in
structural graph theory, as suggested by their use in the study of
claw-free graphs (see Chudnovsky and
Seymour~\cite{chudnovsky.seymour:ClawFree}) and bull-free graphs (see
Chudnovsky~\cite{chudnovsky12}).  About algorithms for Berge
trigraphs, Chudnovsky, Trotignon, Trunck and Vu\v
skovi\'c~\cite{chudnovskyTTV:noBsp} seems to be the only available
reference.

\section{Even pairs: a shorter proof of the SPGT}

An \emph{even pair} in a graph is a pair $\{x,y\}$ of vertices such
that every path between them has even length.  This
notion is involved in algorithmic aspects of perfect graphs and allows
to significantly shorten the proof of the SPGT. 
Meyniel~\cite{meyniel:87} proved the following. 

\begin{theorem}[Meyniel 1987]
  \label{th:meyniel}
  A minimally imperfect graph has no even pair.
\end{theorem}

Therefore, even pairs could be an ingredient of a
useful decomposition theorem for Berge graphs.
Given two vertices $x,y$ in a graph $G$, the operation of
\emph{contracting} them means removing $x$ and $y$ and adding one
vertex with edges to every vertex of $G\setminus \{x,y\}$ that is
adjacent in $G$ to at least one of $x,y$; we denote by $G/xy$ the
graph that results from this operation.  Fonlupt and Uhry
\cite{fonlupt.uhry:82} proved that \emph{if $G$ is a perfect graph and
  $\{x,y\}$ is an even pair in $G$, then the graph $G/xy$ is perfect
  and has the same chromatic number as $G$}.  In particular, given a
$\chi(G/xy)$-colouring $c$ of the vertices of $G/xy$, one can easily
obtain a $\chi(G)$-colouring of the vertices of $G$ as follows: keep
the colour for every vertex different from $x,y$; assign to $x$ and $y$
the colour assigned by $c$ to the contracted vertex.  This idea could
be the basis for a conceptually simple colouring algorithm for Berge
graphs: as long as the graph has an even pair, contract any such pair;
when there is no even pair find a colouring $c$ of the contracted graph
and, applying the procedure above repeatedly, derive from $c$ a
colouring of the original graph.  The algorithm for recognizing Berge
graphs mentioned at the end of the preceding paragraph can be used to
detect an even pair in a Berge graph $G$; indeed, it is easy to see
that two non-adjacent vertices $a,b$ form an even pair in $G$ if and
only if the graph obtained by adding a vertex adjacent only to $a$ and
$b$ is Berge.  Thus, given a Berge graph $G$, one can try to colour its
vertices by keeping contracting even pairs until none can be found.
Then some questions arise: what are the Berge graphs with no even
pair?  What are, on the contrary, the graphs for which a sequence of
even-pair contractions leads to graphs that are trivially easy to
colour?

Bertschi~\cite{bertschi:90} proposed the following definitions.  A
graph $G$ is \emph{even contractile} if either $G$ is a clique or
there exists a sequence $G_0, \ldots, G_k$ of graphs such that
$G=G_0$, for $i=0, \ldots, k-1$ the graph $G_i$ has an even pair
$\{x_i, y_i\}$ such that $G_{i+1}=G_i/x_iy_i$, and $G_k$ is a clique.
A graph $G$ is \emph{perfectly contractile} if every induced subgraph
of $G$ is even contractile.

A simple observation is that odd holes and antiholes of length at
least~5 have no even pairs.  Also some Berge graphs have no even
pairs, such as $\overline{C_6}$, $L(K_{3, 3}) \setminus e$, and more
generally sufficiently connected line graphs of bipartite graphs (see
Hougardy~\cite{hougardy:95} for more about that), and every even
antihole of length at least~6.  In fact every known example of a
Berge graph with no even pair either contains an odd prism or an
antihole.  Odd prisms different from $\overline{C_6}$ have even pair,
but they are not perfectly contractile (any attempt to contract them
leads to $\overline{C_6}$).  This justifies the following definitions.
\emph{Grenoble graphs} are these graphs with no odd holes, no antihole
of length at least 5 and no odd prism.  \emph{Artemis graphs} are
these graphs with no odd holes, no antiholes of length at least 5 and
no prisms.

\begin{figure}
  \center
  \includegraphics{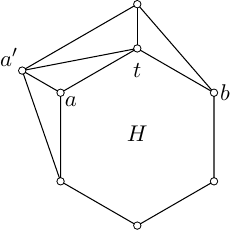}
  \caption{An artemis graph\label{artemis.fig.sansc4}}
\end{figure}

To prove the existence of an even pair in an Artemis graph, an idea is to
consider a shortest hole $H$, and two vertices $u, v$ at distance~2
along $H$, with a common neighbor $t \in V(H)$.  As shown in
Fig.~\ref{artemis.fig.sansc4}, $\{a,b\}$ may fail to be an even pair
because of $a'$ which is the second vertex of a $P_4$ from $a$ to~$b$.
However, one may complain that $H$ is not the `best' hole, the hole
$H'$ obtained by replacing $a$ by $a'$ is better, because $\{a', b\}$
is an even pair of $G$.  In this sense, $a'$ is `better' than $a$.
In the square-free case, Linhares and
Maffray~\cite{linhares.maffray:evenpairsansc4} showed that the
neighborhood of $t$ contains two cliques $A$ and $B$.  On $A$ (and
also on $B$), there exists an order (corresponding to the idea of a
vertex better than another one).  A maximal element in $A$ and a
maximal element in $B$ forms an even pair of $G$.  This method was
extended by Maffray and Trotignon~\cite{nicolas:artemis} to general
Artemis graphs by using the Roussel--Rubio lemma.

\begin{theorem}[Maffray and Trotignon 2005]
  \label{th:artemis}
  Every Artemis graph is perfectly contractile. 
\end{theorem}

This theorem can be transformed into a colouring
algorithm~\cite{nicolas:fastartemis} that generalizes several well
known algorithms for colouring classes of perfect graphs, such as
Meyniel graphs and weakly chordal graphs.  It yields a combinatorial
polynomial time colouring algorithm for perfectly orderable graphs
(none was known before).

\begin{theorem}[L\'ev\^eque, Maffray, Reed and Trotignon 2009]
  \label{th:artemis}
  There exits an algorithm that colours every Artemis graph in time
  $O(n^2m)$. 
\end{theorem}

The proof of the next theorem from~\cite{chudnovsky.seymour:even} uses
some technics of the proof of Theorem~\ref{th:artemis}, is very
readable, and as we will see saves about 50 pages in the proof of the
SPGT.  An \emph{odd wheel} $(C,T)$ in a graph $G$ consists of a hole
$C$ of length at least six, and a nonempty anticonnected subset $T
\subseteq V(G) \setminus V(C)$, such that at least three vertices of
$C$ are $T$-complete, and there is a path $P$ of $C$ with odd length
at least~3, such that its ends are not $T$-complete and all its
internal vertices are $T$-complete.  A \emph{long prism} is a prism
such that at least one the paths has length at least~2.  Let us say
that $G$ is impoverished if $G$ is Berge, and $G$ and $\overline{G}$
both contain no odd wheel, long prism or double diamond.  A
\emph{dominant pair} in $G$ is a pair $(x, y)$ of nonadjacent vertices
such that every other vertex of $G$ is adjacent to at least one of $x,
y$.

\begin{theorem}[Chudnovsky and Seymour, 2009]
  \label{th:pmep}
  If $G$ is impoverished, then either $G$ admits a star cutset or an
  even pair or a dominant pair, or $G$ is a complete graph.
\end{theorem}

It is not very difficult to prove that a minimally imperfect graph has
no dominant pair (see~\cite{chudnovsky.seymour:even}), so it follows
easily by Theorem~\ref{th:meyniel} and Lemma~\ref{th:sc} that every
impoverished graph is perfect.  Since the last 55 pages of the proof
of the SPGT are devoted to finding a skew partition in an impoverished
graph, they are no longer necessary and can be replaced by the 8
pages needed to prove Theorem~\ref{th:pmep}.

A neat generalization of Theorem~\ref{th:artemis} is the following
conjecture (see~\cite{everett.f.l.m.p.r:ep}).

\begin{conjecture}[Everett and Reed]
  Every Grenoble graph is perfectly contractile.
\end{conjecture}

To prove perfection for a class $\cal C$, a statement of the
following form is enough: every graph in $\cal C$ has an even pair, or
its complement has an even pair.  Note that this does not hold for all
Berge graphs as shown by $L(K_{3, 3}\setminus e)$ (see
Fig.~\ref{fig:spo}).  It seems that the only known theorem of this
form is the following from~\cite{figuereido.m.p:bullfree}.  The
\emph{bull} is the graph obtained from the triangle by adding two
pending edges at different vertices.

\begin{theorem}[de Figueiredo, Maffray and Porto, 1997]
  \label{th:bullep}
  If $G$ is a bull-free Berge graph with at least two vertices then at
  least one of $G$ or $\overline{G}$ has an even pair.
\end{theorem}

A graph $G$ is \emph{bipartisan} if it is Berge and contains no
$L(K_{3, 3} \setminus e)$, no double diamond, and none of $G,
\overline{G}$ contains a long prism (a prism is \emph{long} when at
least one of three paths of the prism has length at least~2).  Observe
that every bull-free Berge graph is bipartisan.  An intermediate
result of~\cite{chudnovsky.r.s.t:spgt} is that a bipartisan graph is
bipartite, complement of bipartite, or has a balanced skew partition.
A direct proof of the following conjecture (published
in~\cite{nicolas:cliques} and~\cite{chudnovsky.seymour:even}) would
generalize Theorem~\ref{th:bullep} and could shorten the proof of the
SPGT.

\begin{conjecture}[Maffray, 2002]
If $G$ is  bipartisan  with at least two vertices, then one of $G$,
$\overline{G}$ contains  an even pair.
\end{conjecture}

\subsection*{Further reading}

A good survey on even pairs is Everett, de~Figueiredo, Linhares~Sales,
Maffray, Porto and Reed~\cite{everett.f.l.m.p.r:ep}.  It seems that
even pairs is a conjecture-cornucopia: other conjectures on even pairs
can be found in Burlet, Maffray and Trotignon~\cite{nicolas:cliques}
or in L\'ev\^eque and de Werra~\cite{levequeW12}.  Very short proofs
of the existence of even pairs in classical classes of perfect graphs
(Meyniel graphs and weakly chordal graphs) can be found in Trotignon
and Vu\v skovi\'c~\cite{nicolas.kristina:RR}.

\section{Colouring perfect graphs}
\label{sec:col}

In the 1980s, Gr\"ostchel, Lov{\'a}sz and
Schrijver~\cite{grostchel.l.s:color} devised a polynomial time
algorithm that colours any input perfect graph.  This algorithm relies
on the ellipsoid method, and one may wonder whether a more
combinatorial algorithm exists.  Even if there is no formal definition
of what a \emph{combinatorial} algorithm should be, most
mathematicians agree that an algorithm that relies on graphs searches
and decompositions, and even on classical linear programming, can be
called `combinatorial' and that the ellipsoid method cannot.  So
this question is considered open.  A more precise question is
whether a fast colouring algorithm can be derived from
Theorem~\ref{th:dec}.  We investigate here recent progress in this
direction.

We start with a combinatorial polynomial time colouring algorithm for perfect
graphs, due to Gr\"otschel, Lov\'asz and
Schrijver~\cite{grostchel.l.s:color}, under the assumption that a
subroutine for computing a maximum weighted stable set is available.
Here weights are not essential, because they can be simulated
by replications, but we use them for convenience and because some
subclasses of perfect graphs are not closed under replication.
We suppose that ${\cal C}$ is a subclass of perfect graphs, and that
there is an $O(n^k)$ algorithm ${\cal A}$ that computes a maximum
weighted stable set and a maximum weighted clique for any input graph
in ${\cal C}$ (so $\cal C$ needs not be closed under taking
complement).  Observe that we do not assume that $\cal C$ is closed
under replication or even under taking induced subgraphs.  In what follows,
$n$ denotes the number of vertices of the graph under consideration.

\begin{lemma}
  \label{l:compS}
  There is an algorithm with the following specification:
  \begin{description}
  \item[Input: ] A graph $G$ in ${\cal C}$, and a sequence
    $K_1, \dots, K_t$ of maximum cliques of $G$ where $t\leq n$.
  \item[Output: ] A stable set of $G$ that intersects each $K_i$,
    $i=1, \dots, t$. 
  \item[Running time: ] ${O}(n^k)$
  \end{description} 
\end{lemma}

\begin{proof}
  By $\omega(G)$ we mean here the maximum \emph{cardinality} of a
  clique in~$G$.  Give to each vertex $v$ the weight $y_v= |\{ i; v
  \in K_i \}|$.  Note that this weight is possibly zero.  With
  Algorithm ${\cal A}$, compute a maximum weighted stable set $S$ of~$G$.

  Let us consider the graph $G'$ obtained from $G$ by replicating
  $y_v$ times each vertex $v$.  So each vertex $v$ in $G$ becomes a
  stable set $Y_v$ of size $y_v$ in $G'$ and between two such stable sets
  $Y_u$, $Y_v$ there are all possible edges if $uv\in E(G)$ and
  no edges otherwise.  Note that vertices of weight zero in $G$
  are not in $V(G')$.  Note also that $G'$ may fail to be in ${\cal C}$,
  but it is easily seen to be perfect.  By replicating $y_v$ times
  each vertex $v$ of $S$, we obtain a stable set $S'$ of $G'$ of
  maximum cardinality.

  By construction, $V(G')$ can be partitioned into $t$ cliques of size
  $\omega (G)$ that form an optimal colouring of $\overline{G'}$
  because $\omega(G') = \omega(G)$.  Since by Theorem~\ref{th:pgt}
  $\overline{G'}$ is perfect, $|S'|=t$.  So, in $G$, $S$ intersects
  every $K_i$, $i \in \{ 1, \ldots ,t\}$.
\end{proof}

\begin{theorem}[Gr{\"o}stchel, Lov{\'a}sz and Schrijver 1988]
  \label{th:color}
  There exists an algorithm of complexity $O(n^{k+2})$ whose input is
  a graph from ${\cal C}$ and whose output is an optimal colouring of $G$.
\end{theorem}

\begin{proof}
  As in the proof of Lemma~\ref{l:simpleC}, we only need to show how
  to find a stable set $S$ intersecting all maximum cliques of $G$,
  since we can apply recursion to $G \setminus S$ (by giving weight~0
  to vertices of $S$).  Start with $t=0$. At each iteration, we have a
  list of $t$ maximum cliques $K_1, \ldots, K_t$ and we compute by the
  algorithm in Lemma~\ref{l:compS} a stable set $S$ that intersects
  every $K_i$, $i \in \{ 1, \ldots ,t \}$.  If $\omega (G \setminus S)
  < \omega (G)$ then $S$ intersects every maximum clique, otherwise we
  can compute a maximum clique $K_{t+1}$ of $G \setminus S$ (by giving
  weight~0 to vertices of~$S$).  This will eventually find the desired
  stable set, the only problem being the number of iterations.  We
  show that this number is bounded by $n$.

  Let $M_t$ be the incidence matrix of the cliques $K_1, \dots, K_t$.
  So the columns of $M_t$ correspond to the vertices of $G$ and each
  row is a clique (we see $K_i$ as row vector).  We prove by induction
  that the rows of $M_t$ are independent.  So, we assume that the rows
  of $M_t$ are independent and prove that this holds again for $M_{t+1}$.

  The incidence vector $x$ of $S$ is a solution to $M_tx = \mathbf{1}$
  but not to $M_{t+1}x = \mathbf{1}$.  If the rows of $M_{t+1}$
  are not independent, we have $K_{t+1} = \lambda_1 K_1 + \cdots +
  \lambda_t K_t$.  Multiplying by $x$, we obtain $K_{t+1}x = \lambda_1
  + \cdots + \lambda_t \neq 1$.  Multiplying by $\mathbf{1}$, we
  obtain $\omega = K_{t+1}\mathbf{1} = \lambda_1 \omega + \cdots +
  \lambda_t \omega$, so $\lambda_1 + \cdots + \lambda_t = 1$, a
  contradiction.

  So the matrices $M_1, M_2, \dots$ cannot have more than $n$
  rows.  Hence, there are at most $|V(G)|$ iterations.
\end{proof}

By Theorem~\ref{th:color}, we know that finding a maximum weighted
stable set is enough to solve the colouring problem.  Our question now
is whether Theorem~\ref{th:dec} (or one of its variants from
Section~\ref{sec:trigraphs}) helps to find a maximum weighted stable
set.  It turns out that the question is easy for all basic classes
(for the historical ones, see~\cite{schrijver:opticomb} and for
doubled graphs, see~\cite{chudnovskyTTV:noBsp}).  Also homogeneous pairs
and complement 2-joins are not very hard to handle
(see~\cite{chudnovskyTTV:noBsp}).

For 2-joins, the situation is more complicated because a maximum
stable set may overlap a 2-join in many ways.  To illustrate the
problem, we start with an NP-hardness result.  We define a class
${\cal C}'$ of graphs for which computing a maximum stable set is
NP-hard.  The interesting feature of ${\cal C}'$ is that all graphs in
${\cal C}'$ are decomposable along 2-joins into one bipartite graph
and several gem-wheels where a \emph{gem-wheel} is any graph made of
an induced cycle of length at least 5 together with a vertex adjacent
to exactly four consecutive vertices of the cycle.  Note that a
gem-wheel is a line graph (of a cycle with one chord) and that
computing a maximum weighted stable set in line graph
$G=L(R)$ means computing a maximum weighted matching in $R$, which can
be done by Edmonds's algorithm~\cite{edmonds:ptf}.  Therefore, the
NP-completeness result below shows that being able to decompose along
 2-joins into `easy' graphs is not enough in general to
compute stables sets.

A \emph{flat} path in a graph is path whose internal vertices have
degree~2 (in the graph).  \emph{Extending} a flat path $P = p_1\dots
p_k$ of a graph means deleting the interior vertices of $P$ and adding
three vertices $x, y, z$ and the following edges: $p_1x$, $xy$,
$yp_k$, $zp_1$, $zx$, $zy$, $zp_k$.  By extending a graph $G$ we mean
extending all paths of $\cal M$ where $\cal M$ is a set of
vertex-disjoint flat paths of length at least~3 of $G$.  Class ${\cal
  C}'$ is the class of all graphs obtained by extending 2-connected
bipartite graphs.  From the definition, it is clear that all graphs of
${\cal C}'$ are decomposable along non-path 2-joins.  One leaf of a
decomposition tree will be the underlying bipartite graph.  All the
others leaves will be gem-wheels.

We call \emph{4-subdivision} any graph $G$ obtained from a graph $H$
by subdividing every edge four times.  More precisely, every edge $uv$
of $H$ is replaced by an induced path $u a b c d v$ where $a, b, c, d$
are of degree two.  It is easy to see that $\alpha(G) = \alpha(H) +
2|E(H)|$. This construction, essentially due to
Poljak~\cite{poljak:74}, yields the next result observed by
Naves (see~\cite{nicolas.kristina:2-join}):

\begin{theorem}[Naves, Trotignon and Vu\v skovi\'c 2012]
  \label{th:npHard}
  The problem whose instance is a graph $G$ from $\cal C'$ and an
  integer $k$, and whose question is `Does $G$ contain a stable set
  of size at least~$k$' is NP-complete.
\end{theorem}

\begin{proof}
  Let $H$ be any graph.  First we subdivide 5 times every edge of~$H$.
  So each edge $ab$ is replaced by $P_7 = a  p_1  \dots  p_5
   b$.  The graph $H'$ obtained is bipartite.  Now we build an
  extension $G$ of $H'$ by replacing all the $P_5$'s $p_1  \dots
   p_5$ arising from the subdivisions in the previous step by
  $P_4$'s.  And for each $P_4$ we add a new vertex complete to it and
  we call \emph{apex vertices} all these new vertices.  The graph $G$
  that we obtain is in $\cal C$.  It is easy to see that there exists
  a maximum stable set of $G$ that contain no apex vertex because an
  apex vertex of a maximum stable set can be replaced by one vertex
  of its neighborhood.  So, we call $G'$ the graph obtained from $G$
  by deleting all the apex vertices and see that $\alpha(G') =
  \alpha(G)$.  Also, $G'$ is the 4-subdivision arising from $H$.  So
  from the remark above, maximum stable sets in $H$ and $G$ have sizes
  that differ by $2|E(H)|$.
\end{proof}

We now explain how 2-joins can in fact help to find stable sets
\emph{in Berge graphs}. 
If $(X, Y)$ is a $2$-join of a graph $G$ then let $X_1=X$, $X_2=Y$ and let
$A_1,$ $B_1,$ $C_1,$ $A_2,$ $B_2,$ $C_2$ be as in the definition of a 2-join.
We define $\alpha_{AC} = \alpha(G[{A_1 \cup C_1}])$, $\alpha_{BC} =
\alpha(G[B_1 \cup C_1])$, $\alpha_{C}= \alpha(G[C_1])$ and $\alpha_{X}
= \alpha(G[X_1])$.  Let $w$ be the weight function on $V(G)$.  When $H$
is an induced subgraph of $G$, or a subset of $V(G)$, $w(H)$
denotes the sum of the weights of vertices in $H$.  The following simple
lemma  describes the situation.

\begin{lemma}
  \label{l:4cases}
  Let $S$ be a maximum weighted strong stable set of $G$. Then exactly
  one of the following holds:

  \begin{enumerate}
  \item\label{i:4c1} $S \cap A_1 \neq \emptyset$, $S \cap B_1 =
    \emptyset$, $S\cap X_1$ is a maximum weighted strong stable set of
    $G[A_1 \cup C_1]$ and $w(S \cap X_1) = \alpha_{AC}$;
  \item\label{i:4c2} $S \cap A_1 = \emptyset$, $S \cap B_1 \neq
    \emptyset$, $S\cap X_1$ is a maximum weighted strong stable set of
    $G[B_1 \cup C_1]$ and $w(S \cap X_1) = \alpha_{BC}$;
  \item\label{i:4c3} $S \cap A_1 = \emptyset$, $S \cap B_1 =
    \emptyset$, $S\cap X_1$ is a maximum weighted strong stable set of
    $G[C_1]$ and $w(S \cap X_1) = \alpha_{C}$;
  \item\label{i:4c4} $S \cap A_1 \neq \emptyset$, $S \cap B_1 \neq
    \emptyset$, $S\cap X_1$ is a maximum weighted strong stable set of
    $G[X_1]$ and $w(S \cap X_1) = \alpha_{X}$.
  \end{enumerate}
\end{lemma}

\begin{proof}
  Follows directly from the definition of a $2$-join.
\end{proof}

The next  inequalities are from~\cite{nicolas.kristina:2-join}.  They
say  how  stable sets and $2$-joins overlap in Berge graphs.  

\begin{lemma}[Trotignon and Vu\v skovi\'c 2012]
  \label{l:ineqbasic}
  $0 \leq \alpha_{C} \leq \alpha_{AC}, \alpha_{BC} \leq \alpha_{X}
  \leq \alpha_{AC}+\alpha_{BC}$.
\end{lemma}

\begin{proof}
  The inequalities $0 \leq \alpha_{C} \leq \alpha_{AC}, \alpha_{BC}
  \leq \alpha_{X}$ are trivially true. Let $D$ be a maximum weighted
  stable set of $G[X_1]$.  We have:
  $$
  \alpha_{X} = w(D) = w(D\cap A_1) + w(D\cap (C_1 \cup B_1)) \leq
  \alpha_{AC} + \alpha_{BC}.
  $$
\end{proof}

\begin{lemma}[Trotignon and Vu\v skovi\'c 2012]
  \label{l:ineqOdd}
  If $(X_1, X_2)$ is an odd $2$-join of $G$, then $\alpha_{C}+\alpha_{X}
  \leq \alpha_{AC}+\alpha_{BC}$.
\end{lemma}

\begin{proof}
  Let $D$ be a stable set of $G[X_1]$ of weight $\alpha_{X}$
  and $C$ a stable set of $G[C_1]$ of weight $\alpha_{C}$.  In
  the bipartite graph $G[(C\cup D]$, we denote by $Y_A$ (resp.\
  $Y_B$) the set of those vertices of $C\cup D$ for which there exists
  a path in $G[C \cup D]$ joining them to some vertex of $D\cap
  A_1$ (resp.\ $D \cap B_1$).  Note that from the definition, $D \cap
  A_1 \subseteq Y_A$, $D \cap B_1 \subseteq Y_B$ and there are no
  edges between $Y_A \cup Y_B$ and $(C\cup D)\setminus (Y_A \cup
  Y_B)$.  We claim that $Y_A \cap Y_B= \emptyset$, and $Y_A$ is
  anticomplete to $Y_B$.  Suppose not. Then there exists a
  path $P$ in $G[C\cup D]$ from a vertex of $D \cap A_1$ to a
  vertex of $D \cap B_1$.  We may assume that $P$ is minimal with
  respect to this property, and so the interior of $P$ is in $C_1$;
  consequently $P$ is of even length because $G[C\cup D]$ is
  bipartite.  This contradicts the assumption that $(X_1, X_2)$ is
  odd.  Now we set:

  \begin{itemize}
  \item $Z_A = (D \cap Y_A) \cup (C \cap Y_B) \cup (C \setminus (Y_A \cup
    Y_B))$;
  \item $Z_B = (D \cap Y_B) \cup (C \cap Y_A) \cup (D \setminus (Y_A \cup Y_B)$.
  \end{itemize}

  From all the definitions and properties above, $Z_A$ and $Z_B$ are
  stable sets and $Z_A \subseteq A_1 \cup C_1$ and $Z_B
  \subseteq B_1 \cup C_1$.  So, $\alpha_{C}+\alpha_{X} = w(Z_A) +
  w(Z_B) \leq \alpha_{AC}+\alpha_{BC}$.  
\end{proof}

\begin{lemma}[Trotignon and Vu\v skovi\'c 2012]
  \label{l:ineqEven}
  If $(X_1, X_2)$ is an even $2$-join of $G$, then
  $\alpha_{AC}+\alpha_{BC} \leq \alpha_{C}+\alpha_{X}$.
\end{lemma}

\begin{proof}
  Let $A$ be a stable set of $G[A_1 \cup C_1]$ of weight
  $\alpha_{AC}$ and $B$ a stable set of $G[B_1 \cup C_1]$ of
  weight $\alpha_{BC}$.  In the bipartite graph $G[A\cup B]$, we
  denote by $Y_A$ (resp.\ $Y_B$) the set of those vertices of $A\cup
  B$ for which there exists a path $P$ in $G[A \cup B]$ joining
  them to a vertex of $A \cap A_1$ (resp.\ $B \cap B_1$).  Note that
  from the definition, $A \cap A_1 \subseteq Y_A$, $B \cap B_1
  \subseteq Y_B$, and $Y_A\cup Y_B$ is anticomplete to
  $(A\cup B)\setminus (Y_A \cup Y_B)$.  We claim that $Y_A \cap Y_B =
  \emptyset$ and $Y$ is anticomplete to $Y_B$.  Suppose not,
  then there is a path $P$ in $G[A\cup B]$ from a vertex of $A \cap
  A_1$ to a vertex of $B \cap B_1$.  We may assume that $P$ is minimal
  with respect to this property, and so the interior of $P$ is in
  $C_1$; consequently it is of odd length because $G (A\cup B)$ is
  bipartite.  This contradicts the assumption that $(X_1, X_2)$ is
  even.  Now we set:

  \begin{itemize}
  \item $Z_D = (A \cap Y_A) \cup (B \cap Y_B) \cup (A \setminus (Y_A \cup Y_B))$;
  \item $Z_C = (A \cap Y_B) \cup (B \cap Y_A) \cup (B \setminus (Y_A \cup Y_B))$.
  \end{itemize}

  From all the definitions and properties above, $Z_D$ and $Z_C$ are
  stable sets and $Z_D \subseteq X_1$ and $Z_C \subseteq C_1$.
  So, $\alpha_{AC}+\alpha_{BC} = w(Z_C) + w(Z_D) \leq
  \alpha_{C}+\alpha_{X}$.  
\end{proof}

The two lemmas above allow to construct blocks of decomposition of a
2-join that preserve being Berge and allow to keep track of $\alpha$
(see~\cite{chudnovskyTTV:noBsp} for the precise definition of the
blocks).  Interestingly, 2-joins are used to compute $\alpha$ in other
classes of graphs \nocite{d}(while they seem to be hard to use in
general): in claw-free graphs (see Faenza, Oriolo and
Stauffer~\cite{faenzaOrioloStauffer:clawFree}), and in even-hole-free
graphs with no star cutsets (see Trotignon and Vu\v
skovi\'c~\cite{nicolas.kristina:2-join}).

In \cite{chudnovskyTTV:noBsp}, the inequalities above are used to prove
the following.  Many technicalities are needed, and some of them come
from the fact that the blocks of decomposition for 2-joins that keep
track of $\alpha$ do not preserve being balanced-skew-partition-free.
Also, it is proved in~\cite{chudnovskyTTV:noBsp} that for Berge graphs
with no balanced skew partition, there exist \emph{extreme}
decompositions, that are decompositions such that one of the block is
basic.  These are very convenient for proofs by induction.  To handle
all these technicalities, it is convenient (if not mandatory) to work
with trigraphs, but here we state the result for graphs.

\begin{theorem}[Chudnovsky, Trotignon, Trunck and Vu\v skovi\'c 2012]
  There exists an $O(n^7)$ time algorithm whose input is a Berge graph
  with no balanced skew partition and whose output is a maximum
  weighted stable set of $G$ and a colouring of $G$.
\end{theorem}

So far, no one knows how skew partitions could be handled to provide a
polynomial colouring algorithm.  One might think that a Berge graph uniquely
decomposable with a balanced skew partition must have an even pair.
To support this idea, Chudnovsky and
Seymour~\cite{chudnovsky.seymour:k4freeEven} studied the structure of
Berge graphs with no $K_4$ and no even pair.  They describe them quite
precisely.

\begin{theorem}[Chudnovsky and Seymour 2012]
  If $G$ is a 3-connected $K_4$-free Berge graph with no even pair,
  and with no clique cutset, then one of $G$, $\overline{G}$ is the
  line graph of a bipartite graph.
\end{theorem}

This theorem was generalized by Zwols~\cite{zwols:10} to $\{K_4,
\text{odd holes}\}$-free graphs (where non-perfect exceptions exist:
$\overline{C_7}$, and a special graph on eleven vertices).
Unfortunately, it seems that this theorem does not generalize to
larger values of $\omega$, as shown by the WBGKSF, a graph $G$
(discovered by Chudnovsky and Seymour, unpublished) represented on
Fig.~\ref{wkbgsf}.  Every edge in $G$ is the middle edge of a $P_4$.
This means that in the complement, every pair of non-adjacent vertices
can be linked by a $P_4$.  Therefore, $\overline{G}$ has no even
pairs.  However, $G$ is perfect, and the balanced skew partition is
the only outcome of Theorem~\ref{th:dec} satisfied by $G$.  In fact,
$G$ has even pairs, so the next conjecture (unpublished) could still
be true. It is quite challenging, but it is not clear whether it would
help to colour perfect graphs because even pairs in the complement do
not seem usable.

\begin{figure}
  \center
  \includegraphics{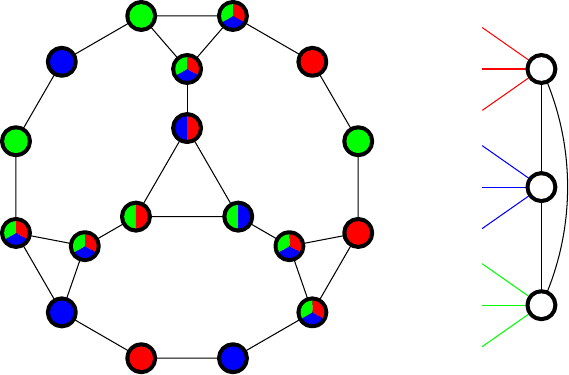}
  \caption{The WBGKSF (Worst Berge Graph Known So Far)\label{wkbgsf}.
    Red (resp.\ green, blue) edges go to red (resp.\ green, blue)
    vertices.}
\end{figure}

\begin{conjecture}[Thomas, 2002]
  If a Berge graph $G$ is uniquely decomposable by a balanced skew
  partitition (so $G$ is not basic, has no 2-join and no complement
  2-join), then one of $G$ or $\overline{G}$ has an even pair.
\end{conjecture}

I feel the next two questions as the most important ones about perfect
graphs.

\begin{question}
  Describe the structure of Berge graphs with no even pairs. 
\end{question}

\begin{question}
  Describe a combinatorial polynomial time algorithm that colours every
   Berge graphs.  
\end{question}

\section*{Acknowledgement}
\addcontentsline{toc}{section}{Acknowledgement}

Thanks to Maria Chudnovsky, Au\'elie Lagoutte, Irena Penev, Fr\'ed\'eric Maffray and
Robin Wilson for several suggestions that improved this survey.

\bibliographystyle{plain}

\bibliography{../../Bibliographie/articles}

\addcontentsline{toc}{section}{References}
\end{document}